\documentclass[reqno]{amsart}
\usepackage{amsmath}
\usepackage{amsthm}
\usepackage{amssymb}

\usepackage[utf8]{inputenc}
\usepackage[T1]{fontenc}
\usepackage{lmodern}
\usepackage{eucal}
\pdfoutput=1 
\usepackage{microtype}
\usepackage{url}
\usepackage{enumitem}
\usepackage{tikz-cd}
\usepackage{bbm}
\usepackage{longtable,booktabs}

\usepackage[letterpaper, left = 1.5in, right = 1.5in, top = 1.5in, bottom = 1.5in]{geometry}

\swapnumbers

\usepackage{hyperref}

\usepackage{xcolor}
\hypersetup{
    colorlinks,
    linkcolor={red!50!black},
    citecolor={blue!50!black},
    urlcolor={blue!80!black}
}

\usepackage[capitalise,nameinlink]{cleveref}
\crefname{subsection}{Subsection}{Subsections}
\crefname{subsubsection}{Subsubsection}{Subsubsections}

\theoremstyle{definition}
\newtheorem{theorem}{Theorem}[subsection]
\newtheorem{defn}[theorem]{Definition}
\newtheorem{example}[theorem]{Example}
\newtheorem{ex}[theorem]{Example}
\newtheorem{cor}[theorem]{Corollary}
\newtheorem{lemma}[theorem]{Lemma}
\newtheorem{prop}[theorem]{Proposition}
\newtheorem{rmk}[theorem]{Remark}

\newtheorem*{rmk*}{Remark}
\newtheorem*{ex*}{Example}
\newtheorem*{theorem*}{Theorem}
\newtheorem*{defn*}{Definition}

\newcommand{\bbZ}{\mathbb{Z}}
\newcommand{\bbG}{\mathbb{G}}
\newcommand{\bbH}{\mathbb{H}}
\newcommand{\bbF}{\mathbb{F}}

\newcommand{\bbP}{\mathbb{P}}
\newcommand{\bbE}{\mathbb{E}}
\newcommand{\bbQ}{\mathbb{Q}}
\newcommand{\bbR}{\mathbb{R}}

\newcommand{\bbA}{\mathbb{A}}

\newcommand{\bbC}{\mathbb{C}}

\newcommand{\bbM}{\mathbb{M}}

\newcommand{\calM}{\mathcal{M}}
\newcommand{\calN}{\mathcal{N}}
\newcommand{\calO}{\mathcal{O}}
\newcommand{\calC}{\mathcal{C}}
\newcommand{\calD}{\mathcal{D}}

\newcommand{\calF}{\mathcal{F}}
\newcommand{\calJ}{\mathcal{J}}

\newcommand{\calL}{\mathcal{L}}

\newcommand{\calA}{\mathcal{A}}
\newcommand{\calP}{\mathcal{P}}

\newcommand{\calH}{\mathcal{H}}

\newcommand{\Sp}{\mathcal{S}\mathrm{p}}

\newcommand{\Gpd}{\mathcal{G}\mathrm{pd}}

\newcommand{\Ab}{\mathcal{A}\mathrm{b}}
\newcommand{\Fun}{\mathrm{Fun}}

\newcommand{\Ind}{\operatorname{Ind}}

\newcommand{\Comm}{\mathcal{C}\mathrm{omm}}
\newcommand{\Glob}{\mathcal{G}\mathrm{lob}}
\newcommand{\Tow}{\mathcal{T}\mathrm{ow}}
\newcommand{\Thick}{\operatorname{Thick}}

\newcommand{\Map}{\operatorname{Map}}

\newcommand{\Aut}{\operatorname{Aut}}

\newcommand{\Fib}{\operatorname{Fib}}

\newcommand{\Ext}{\operatorname{Ext}}
\newcommand{\Cof}{\operatorname{Cof}}
\newcommand{\coker}{\operatorname{Coker}}
\newcommand{\colim}{\operatorname*{colim}}

\newcommand{\im}{\operatorname{Im}}
\renewcommand{\ker}{\operatorname{Ker}}

\newcommand{\gr}{\operatorname{gr}}
\newcommand{\Sq}{\mathrm{Sq}}

\newcommand{\Spf}{\operatorname{Spf}}

\newcommand{\res}{\operatorname{res}}
\newcommand{\tr}{\operatorname{tr}}

\newcommand{\Th}{\operatorname{Th}}
\newcommand{\Cl}{\operatorname{Cl}}

\newcommand{\gl}{\mathrm{gl}}
\newcommand{\rk}{\operatorname{rk}}

\newcommand{\epsilontilde}{\tilde{\epsilon}}

\newcommand{\h}{\mathrm{h}}

\newcommand{\cl}{\mathrm{cl}}

\newcommand{\fast}{\mathrm{fast}}
\newcommand{\nil}{\mathrm{nil}}

\newcommand{\bs}{{-}}
 
\newcommand{\ul}{\underline}
\newcommand{\ol}{\overline}
\def\hyp{{\hbox{-}}}

\newcommand\xqed[1]{%
  \leavevmode\unskip\penalty9999 \hbox{}\nobreak\hfill
  \quad\hbox{#1}}
\newcommand\tqed{\xqed{$\triangleleft$}}

\makeatletter
\DeclareRobustCommand{\tvdots}{%
  \vbox{\baselineskip4\p@\lineskiplimit\z@\kern0\p@\hbox{.}\hbox{.}\hbox{.}}}
\makeatother

\makeatletter
\@namedef{subjclassname@2020}{\textup{2020} Mathematics Subject Classification}
\makeatother

\begin{document}

\title{Total power operations in spectral sequences}

\author{William Balderrama}


\begin{abstract}
We describe how power operations descend through homotopy limit spectral sequences. We apply this to describe how norms appear in the $C_2$-equivariant Adams spectral sequence, to compute norms on $\pi_0$ of the equivariant $KU$-local sphere, and to compute power operations for the $K(1)$-local sphere. An appendix contains material on equivariant Bousfield localizations which may be of independent interest. 
\end{abstract}

\maketitle

\tableofcontents

\section{Introduction}\label{sec:intro}

Let $G$ be a finite group. The best analogue of a commutative ring in the context of $G$-equivariant stable homotopy theory is that of a \textit{$G\hyp\bbE_\infty$ ring}, or $G$-equivariant commutative ring spectrum in the sense of \cite{hillhopkinsravenel2016nonexistence}. If $R$ is a $G\hyp\bbE_\infty$ ring, then not only is $R$ equipped with the usual ring operations of addition and multiplication, but also with \textit{multiplicative norms}
\[
N_K^G R \rightarrow R
\]
for all subgroups $K\subset G$, reflecting a higher form of commutativity present on $R$. Here, $N_K^G$ is the Hill--Hopkins--Ravenel norm \cite{hillhopkinsravenel2016nonexistence}; informally, $N_K^G R = R^{\otimes G/K}$, with equivariance intertwining the action of $K$ on $R$ and the action of $G$ on $G/K$.

This additional structure is reflected in algebra. If $R$ is a $G\hyp\bbE_\infty$ ring, then the collection $\ul{\pi}_0 R = \{\pi_0^K R : K\subset G \}$ carries the rich algebraic structure of a \textit{$G$-Tambara functor} \cite{tambara1993multiplicative} \cite{brun2007witt}. This means that, in addition to the linear structure of restrictions, transfers, and products, one has multiplicative but generally \textit{nonadditive} norm maps
\[
N_K^G\colon \pi_0^K R\rightarrow\pi_0^G R,
\]
interacting with the linear structure in rich ways. More generally, norms exist outside degree $0$ as maps
\[
P_\alpha\colon \pi_\alpha^K R\rightarrow \pi_{\Ind_K^G\alpha}^G R
\]
defined for all subgroups $K\subset G$ and virtual orthogonal representations $\alpha\in RO(K)$. This rich algebraic structure has seen extensive study over the past decade, e.g.\ \cite{strickland2012tambara, nakaoka2012ideals, ullman2013tambara, hill2017andre, blumberghill2018incomplete, angeltveitbohmann2018graded, hillmazur2019equivariant}. However, despite this wealth of theoretical work, relatively few specific computations are available, outside certain well-behaved cases. One need not go all the way to equivariant homotopy theory to see this: ordinary $\bbE_\infty$ rings already carry power operations, but relatively few computations are available, outside the most well behaved examples of $\bbE_\infty$ rings in positive characteristic and certain complex-oriented theories.

Consider the problem of computing just the groups $\ul{\pi}_\star R$. The homotopy theorist's tools of choice for such computations are a wide array of \textit{spectral sequences}, which arise whenever one has a way of building $R$ out of simpler pieces. In some cases, these simpler pieces may even be simple enough that one can understand their norms. This leads to the question: how can we take this information and descend it through the spectral sequence?

Norms in spectral sequences have been considered previously, such as in the context of the slice spectral sequence in \cite[Section I.5]{ullman2013regular} and \cite[Section 4]{hillhopkinsravenel2017slice}. Our own interest is in situations that are orthogonal to this; in short, in spectral sequences where Euler classes are detected on the $0$-line. Moreover, we care not just about norms of $G\hyp\bbE_\infty$ spectra, but also other operations of a similar nature, such as power operations for ordinary $\bbE_\infty$ rings \cite{brunermaymccluresteinberger1986hinfinity} and power operations in the global equivariant context \cite{schwede2018global,stahlhauer2021ginfty}.

This paper describes how such operations may be computed in \textit{homotopy limit spectral sequences} (HLSSs), such as generalized Adams spectral sequences and homotopy fixed point spectral sequences. We then give applications which demonstrate how this plays out in practice. In fact these applications, described below in \cref{ssec:applications}, might be considered the core of the paper, although it is the tools used which seem more widely applicable. Let us describe these in the context of equivariant norms as above.

Let $R\colon\calJ\rightarrow\Comm_G$ be a diagram of $G\hyp\bbE_\infty$ ring spectra. From the underlying diagram of $G$-spectra, one may produce for all $K\subset G$ and $\alpha\in RO(K)$ an HLSS which we shall index as
\[
{}^K E_2^{s+\alpha,t+\alpha} = H^{t-s}(\calJ;\pi_{t+\alpha}^K R)\Rightarrow \pi_{s+\alpha}^K \lim_{j\in\calJ}R(j).
\]
Here, we have written $H^n(\calJ;\bs)$ for the $n$th right derived functor of $\lim_{j\in\calJ}\colon\Fun(\calJ,\Ab)\rightarrow\Ab$. For example, when $\calJ = \Delta$, this is the usual spectral sequence of a cosimplicial object.

Write $Z_r^{s,t}$ and $B_r^{s,t}$ for the $r$-cycles and $r$-boundaries of this spectral sequence, and write $a_K^G \in \pi_{1-\bbR[G/K]}^G S_G$ for the class represented by the inclusion of fixed points $\smash{S^1 \rightarrow S^{\bbR[G/K]}}$.

\begin{theorem}[\cref{ssec:powhlss}]\label{thm:main}
${}$
\begin{enumerate}
\item The composite $a_K^G P_\alpha$ is additive. In particular, for $t\geq s\geq 0$ there are induced maps
\[
Q_\alpha = (a_K^G)^tP_{t+\alpha}\colon {}^K E_2^{s+\alpha,t+\alpha}\rightarrow {}^G E_2^{s+\Ind_K^G\alpha,t+\Ind_K^G\alpha}.
\]
\item $Q_\alpha({}^K Z_r^{s+\alpha,t+\alpha})\subset {}^G Z_r^{s+\Ind_K^G\alpha,t+\Ind_K^G\alpha}$ for $t\geq s\geq 0$.
\item $Q_\alpha({}^K B_r^{s+\alpha,t+\alpha})\subset {}^G B_r^{s+\Ind_K^G\alpha,t+\Ind_K^G\alpha}$ for $t\geq s\geq -1$.
\item For $x\in {}^K E_r^{s+\alpha,t+\alpha}$ with $s\geq 0$, we have
\[
d_r(Q_\alpha(x)) = \begin{cases}
Q_\alpha(d_r(x)), & t \geq 1; \\
Q_\alpha(d_r(x)) + c(d_r(x),x), & t=s=0;
\end{cases}
\]
where $c$ is related to the addition formula for $P_\alpha$. For example, when $K = e$ and $G = C_2$, we have $c(d_r(x),x) = \tr(d_r(x) \cdot \ol{x})$ with $\ol{x}$ the involution applied to $x$.
\item If $x\in E_2^{\alpha,t+\alpha}$ is a permanent cycle detecting $f\in \pi_{\alpha}^K \lim_{j\in\calJ}R(j)$, then the permanent cycle $Q_\alpha(x)$ detects $P_\alpha(f)$ modulo classes in higher filtration.
\tqed
\end{enumerate}
\end{theorem}

Informally, $P_\alpha$ is modeled in filtration $t$ by $(a_K^G)^t P_{t+\alpha}$. This is immediately applicable to computations, and we give applications below. As usual, the development was the other way around: we found ourselves with various computations we realized we could carry out, and questions we could answer, if only we had some theorem along these lines. It was clear from the start that such a theorem should follow by a consideration of the space-level norm
\begin{equation}\label{eq:spacenorm}
\calP_\alpha\colon \Map_{\Sp^K}(S^\alpha,\res^G_K R)\rightarrow \Map_{\Sp^G}(S^{\Ind_K^G\alpha},R),
\end{equation}
and in fact most of \cref{thm:main} does follow quickly from an inspection of \cref{eq:spacenorm}, the main observation being that $\pi_t\calP_\alpha = (a_K^G)^t P_{t+\alpha}$. The bulk of the work in the proof of \cref{thm:main} stems from the additional care needed to handle what happens on the fringe; for example, to describe $d_2(P_\alpha(x))$ for $x\in H^0(\calJ;\pi_\alpha^K R)$. Although the applications we give below do not need this more refined information, we expect it will be useful in future work.

\subsection{Applications}\label{ssec:applications}

Let us now describe applications. We begin with an application to the $C_2$-equivariant Adams spectral sequence. Let 
\[
\calA^\cl = \pi_\ast (H\bbF_2\otimes H\bbF_2),\qquad \calA^{C_2} = \pi_\star^{C_2} (H\bbF_2^{C_2} \otimes H\bbF_2^{C_2})
\]
denote the classical and $C_2$-equivariant dual Steenrod algebras, and write \[
\Ext_\cl = H^\ast(\calA^\cl),\qquad \Ext_{C_2} = H^\ast(\calA^{C_2})
\]
for their cohomology, serving as the $E_2$-pages of the classical and $C_2$-equivariant Adams spectral sequences \cite[Section 6]{hukriz2001real} \cite{guillouhillisaksenravenel2020cohomology}. Algebraically, the latter is of the form
\[
\Ext_{C_2} \cong \Ext_\bbR \oplus \Ext_{NC}.
\]
Here, $\Ext_\bbR\subset\Ext_{C_2}$ is the cohomology of the $\bbR$-motivic Steenrod algebra, the inclusion of which is compatible with Adams differentials, and $\Ext_{NC}$ is some other summand. Let $\rho = a_e^{C_2} \in \pi_{-\sigma}S_{C_2}$ denote the Euler class of the sign representation. By \cite[Theorem 4.1]{duggerisaksen2017low}, there is an isomorphism
$
\Ext_{\bbR}[\rho^{-1}] \cong \Ext_\cl[\rho^{\pm 1}],
$
with a suitable shift in degrees. In fact, the proof gives rise to an explicit splitting
\begin{equation}\label{eq:c2split}
\Ext_{C_2} \cong \Ext_\cl[\rho]\oplus\Ext_{\bbR}^{\rho\hyp\text{torsion}} \oplus \Ext_{NC},
\end{equation}
where the copy of $\Ext_\cl$ is given as follows. Write
\[
\calA^\cl = \bbF_2[\xi_1,\xi_2,\ldots],\quad \calA^{\bbR} = \bbF_2[\tau,\rho][\xi_1,\xi_2,\ldots,\tau_0,\tau_1,\ldots]/(\tau_i^2 + \tau \xi_{i+1} + \rho(\tau_0 \xi_{i+1} + \tau_{i+1})).
\]
Then the map
\begin{equation}\label{eq:pa}
P\colon \calA^\cl\rightarrow \calA^{\bbR}\subset\calA^{C_2},\qquad P(\xi_i) = \xi_i
\end{equation}
is a map of Hopf algebroids, and the induced map $P\colon \Ext_\cl\rightarrow\Ext_{C_2}$ picks out the copy of $\Ext_\cl$ in \cref{eq:c2split}.

The isomorphism $\Ext_\bbR[\rho^{-1}]\cong \Ext_\cl[\rho^{\pm 1}]$ extends to an isomorphism $\Ext_{C_2}[\rho^{-1}]\cong \Ext_\cl[\rho^{\pm 1}]$, and both of these isomorphisms have a direct geometric interpretation: the first models real Betti realization, and the second models taking geometric fixed points. On the other hand, the map $P$ appears at first glance to be purely algebraic: for example, it does not preserve permanent cycles. It turns out to have the following geometric significance.

\begin{theorem}[\cref{sec:c2}]\label{thm:c2}
Let $x\in \Ext_\cl$ be a class in filtration $f$. If $x$ survives to the $E_{r}$-page, then $\rho^f P(x)\in\Ext_{C_2}$ survives to the $E_{r}$-page, and
\[
d_{r}(\rho^f P(x)) = \rho^{f+r-1}P(d_{r}(x)).
\]
Moreover, if $x$ is a permanent cycle detecting $\alpha\in \pi_n S$, then the permanent cycle $\rho^fP(x) \in\Ext_{C_2}$ detects the symmetric square $\Sq(\alpha) \in \pi_{n(1+\sigma)}S_{C_2}$.
\qed
\end{theorem}

\cref{thm:c2} is not surprising, given the general shape of $\Ext_{C_2}$. If $x\in\Ext_\cl$ detects $\alpha\in\pi_n S$, then as the geometric fixed points of $\Sq(\alpha)$ are $\alpha$, one finds that $\Sq(\alpha)$ is detected by some preimage of $\alpha$ under the localization $\Ext_{C_2}\rightarrow\Ext_{C_2}[\rho^{-1}]\cong\Ext_\cl[\rho^{\pm 1}]$. If $x$ is in filtration $f$, then this indicates that $\Sq(\alpha)$ is detected by $\rho^f P(x)$ plus possible $\rho$-torsion error terms. \cref{thm:c2} says that in fact $\Sq(\alpha)$ is detected by $\rho^f P(x)$ on the nose, and describes what happens when $x$ is not a permanent cycle.
The proof amounts to relating \cref{eq:pa} to the norms on $\ul{\pi}_\star (H\bbF_2^{C_2}\otimes H\bbF_2^{C_2})$, and then applying \cref{thm:main}.

\begin{example}\label{ex:sqeta}
We have $\rho P(h_1) = \rho h_2$, and thus $\Sq(\eta_{\cl})$ is detected by the same class detecting $\rho \nu_{C_2}$, where $\eta_\cl$ is the nonequivariant complex Hopf fibration and $\nu_{C_2}$ is the $C_2$-equivariant quaternionic Hopf fibration. As $\Sq(\eta_{\cl})$ must also lift $\eta_\cl^2$, by consulting the tables in \cite{duggerisaksen2017low} and using the fact that $\pi_\star S_{C_2}$ agrees with $\pi_{\ast,\ast}S_\bbR$ in this range, we find that the only possibility is 
\[
\Sq(\eta_{\cl}) = \eta_{\cl}\eta_{C_2} + \rho \nu_{C_2}.
\]
This was originally computed by Araki--Iriye \cite[Theorem 10.12]{arakiiriye1982equivariant}, and its computation via \cref{thm:c2} can be considered overkill: as soon as one knows $\pi_\star^{C_2}S_{C_2}$ in these degrees, $\Sq(\eta_\cl)$ is determined by the fact that it lifts $\eta_\cl^2$ and has geometric fixed points $\eta_\cl$.
\tqed
\end{example}

\begin{ex}
Consider $\rho P(h_3) = \rho h_4$. As $h_3$ is a permanent cycle, it follows that $\rho h_4$ is a permanent cycle, as was first shown by Belmont--Isaksen \cite{belmontisaksen2020rmotivic}. Moreover, $\rho h_4$ detects $\Sq(\sigma)$, a fact closely related to the Mahowald invariant $R(\sigma) = \sigma^2$. This was observed in \cite[Theorem 7.4.7]{balderramaculverquigley2021motivic}, which was one of the inspirations for \cref{thm:c2}. This example illustrates that the additional $\rho$'s in \cref{thm:c2} are necessary: $h_4 = P(h_3)$ itself supports the differential $d_2(h_4) = h_0 h_3^2$, and $\rho h_4 = \rho P(h_3)$ is not divisible by $\rho$ on the $E_3$-page.
\tqed
\end{ex}

Our next applications are to power operations in the context of chromatic homotopy theory, at chromatic height $1$. We begin with the following. In recent work, Bonventre--Guillou--Stapleton have shown that if $G$ is an odd $p$-group, then there is an isomorphism
\[
\ul{\pi}_0 L_{KU_G}S_G\cong \ul{R}\bbQ \otimes \pi_0 L_{KU}S = \ul{R}\bbQ[\epsilon]/(2\epsilon,\epsilon^2)
\]
of Green functors \cite[Theorem 1.1, Proposition 6.7]{bonventreguilloustapleton2022kug}. Here, $L_{KU_G}S_G$ is the localization of the $G$-equivariant sphere spectrum with respect to $G$-equivariant $K$-theory, and $\ul{R}\bbQ$ is the Green functor whose value at $K\subset G$ is the rational representation ring of $K$. They also verify that $KU_G$-localization preserves $G\hyp\bbE_\infty$ structures for $G$ an odd $p$-group. This gives $\ul{\pi}_0 L_{KU_G}S_G$ the structure of a Tambara functor, but they are only able to determine its norms in the case where $G = C_{p^n}$ is cyclic \cite[Proposition 10.6]{bonventreguilloustapleton2022kug}. \cref{thm:main} allows us to directly compute norms in contexts like this, and in the end we find the following.

\begin{theorem}[\cref{ssec:ku}]\label{thm:ku}
Fix an odd $p$-group $G$ and subgroup $K \subset G$. Let $\widetilde{\bbQ}[G/K] = \coker(\bbQ\rightarrow\bbQ[G/K])$ be the reduced permutation representation of the $G$-set $G/K$, and define
\[
e(G/K) = \sum_n (-1)^n \Lambda^n(\widetilde{\bbQ}[G/K]) \in R\bbQ(G).
\]
Then the norm
\[
N_K^G\colon R\bbQ(K)[\epsilon]/(2\epsilon,\epsilon^2) \rightarrow R\bbQ(G)[\epsilon]/(2\epsilon,\epsilon^2)
\]
arising from the $G\hyp\bbE_\infty$ structure of $L_{KU_G}S_G$ satisfies
\[
N_K^G(\epsilon) = e(G/K)\cdot\epsilon.
\]
In particular, if $K\subset G$ is normal, then $N_K^G(\epsilon)\neq 0$ if and only if $G/K$ is cyclic, in which case $N_K^G(\epsilon) = \widetilde{\bbQ}[G/N]\cdot \epsilon$ where $N\subset G$ is the unique subgroup of index $p$ containing $K$.\tqed
\end{theorem}

The proof of \cref{thm:ku} amounts to using \cref{thm:main} to show that $N_K^G(\epsilon)$ is detected in the $KU_G$-based Adams spectral sequence by $e(G/K)\cdot \epsilon$. In fact this is true for an arbitrary finite group $G$ and subgroup $K\subset G$, not just for odd $p$-groups. For this and other reasons, the correct context for our computation is not $G$-equivariant homotopy theory for any particular group $G$, but rather \textit{global} equivariant homotopy theory.

Let $\Glob$ be the category of global equivariant spectra with respect to the family of finite groups, and let $\textbf{KU}$ the global spectrum of equivariant $K$-theory, both as developed by Schwede in \cite{schwede2018global}. There are forgetful functors $U_G\colon\Glob\rightarrow\Sp^G$ satisfying $U_G\textbf{KU} = KU_G$; as far as we are concerned, this can be treated as the definition of $KU_G$. In order to ensure compatibility between statements made in the global context and statements made in the context of \cite{bonventreguilloustapleton2022kug}, we prove the following.

\begin{prop}[\cref{prop:ku1proof}]\label{prop:kul}
Let $\Glob_{\nil}$ be the category of global equivariant spectra with respect to the family of finite nilpotent groups
\begin{enumerate}
\item Bousfield localization in $\Glob_{\nil}$ with respect to $\textbf{KU}$ is smashing, agrees with $\textbf{KU}$-nilpotent completion, and preserves ultracommutative ring spectra;
\item If $G$ is nilpotent, then Bousfield localization in $\Sp^G$ with respect to $KU_G$ is smashing, agrees with $KU_G$-nilpotent completion, and preserves $G\hyp\bbE_\infty$ ring spectra;
\item In particular, $U_G L_{\textbf{KU}}S\simeq L_{KU_G}U_G S_G$ for $G$ nilpotent.
\tqed
\end{enumerate}
\end{prop}

The proof of \cref{prop:kul} requires some general theory regarding equivariant Bousfield localizations. This theory is interesting in its own right, and also applies to other examples of interest in chromatic equivariant homotopy theory. For this reason, we have separated out our discussion of equivariant Bousfield localizations into \cref{app:bous}, which may be read independently of the rest of the paper. 

Now let us return to considering power operations. Observe that \cref{thm:ku} is a genuinely integral result, mixing $2$-primary homotopy with odd-primary equivariance. If instead of working integrally we work $K(1)$-locally, then equivariant norms amount to ordinary $K(1)$-local power operations (see \cref{rmk:k1norm}). In \cite{hopkins2014k1}, Hopkins explains how one may use $K(1)$-local splittings
\begin{equation}\label{eq:k1split}
L_{K(1)}B\Sigma_{p+}\simeq S_{K(1)}\oplus S_{K(1)}
\end{equation}
to define the structure of a \textit{$\theta$-ring} on $\pi_0$ of an arbitrary $K(1)$-local $\bbE_\infty$ ring spectrum (see \cref{rmk:theta}). At $p=2$, the $\theta$-ring structure of $\pi_0 S_{K(1)} = \bbZ_2[\epsilon]/(2\epsilon,\epsilon^2)$ has been clarified only recently by Carmeli--Yuan \cite[Theorem 5.4.8]{carmeliyuan2021higher}, who prove that $\theta(\epsilon) = \epsilon$.

The story should not stop with $\pi_0$. However, the picture quickly becomes less clear, as there is no analogue of \cref{eq:k1split} for $L_{K(1)}(S^{2n})^{\otimes p}_{\h \Sigma_p}$ when $n \neq 0$. One of the original motivations for this paper was a desire to be able to compute with these more complicated examples, where power operations cannot be described as some clean algebraic object, such as a $\theta$-ring. Using a suitable variant of \cref{thm:main}, we carry out the following computation.

\begin{theorem}\label{thm:sk1}
Let $S_{K(1)} = L_{KU/(p)}S$. Then the $p$th total power operation
\[
P\colon [S^n,S_{K(1)}]\rightarrow [(S^n)^{\otimes p}_{\h \Sigma_p},S_{K(1)}]
\]
is as given in \cref{thm:ppow} for $p$ odd and \cref{thm:2pow} for $p$ even, modulo a certain indeterminacy at $p=2$ when $n \equiv 1 \pmod{8}$ and $n\neq 1$.
\tqed
\end{theorem}

\subsection{Organization}

This paper is organized as follows. In \cref{sec:unstablenaturality}, we study the naturality of the HLSS of a diagram of spectra with respect to its underlying diagram of pointed spaces. This analysis is well-suited for any homotopy operations obtained from pointed functors between stable categories, and in \cref{sec:pow} we make this explicit in the case of the $m$-fold smash power functors $\bbP^m\colon \Sp^G\rightarrow\Sp^{\Sigma_m\wr G}$.

Both sections culminate in \cref{ssec:powhlss}, which puts everything together into a form suitable for applications, including \cref{thm:main} and variants. The reader interested in the applications may wish to start here.

We then give the promised applications. In \cref{sec:c2}, we prove \cref{thm:c2}; in \cref{sec:kug}, we prove \cref{thm:ku}; and in \cref{sec:k1}, we carry out the computation of \cref{thm:sk1}. In \cref{app:bous}, we give some material on equivariant localizations, including \cref{prop:kul}.

\subsection{Acknowledgements}

It is a pleasure to thank Charles Rezk, Nick Kuhn, Christian Carrick, Jeremy Hahn, and Tomer Schlank for enlightening discussions about power operations, character theory, Bousfield localizations, chromatic convergence, and $G$-spectra. Moreover, we thank Mike Hill for pointing out an error in our original proof of \cref{thm:c2}.

\section{Unstable naturality of the homotopy limit spectral sequence}\label{sec:unstablenaturality}

This section studies the naturality of stable HLSSs with respect to unstable maps. We begin by recalling the construction of the HLSS in the form most convenient to us in \cref{ssec:tower} and \cref{ssec:shfpss}. In \cref{ssec:unstable}, we consider the analogous unstable construction, and in \cref{ssec:stableunstable} we compare the two. We state and prove the main naturality theorem, \cref{thm:unstablenaturality}, in \cref{ssec:unstablenaturality}.

In some sense, this material should be known to those who have worked with extended homotopy spectral sequences in the sense of Bousfield--Kan \cite[Ch.\ IX, \S 4]{bousfieldkan1972homotopy}. The main naturality theorem holds by construction, and most of our work in this section is to recall enough of the construction that we may be sure of this. Moreover, we package this unstable information entirely into the context of ordinary spectral sequences, thereby removing any need to contend with the extended spectral sequences lurking in the background.

\subsection{The spectral sequence of a tower}\label{ssec:tower}

Let
\begin{center}\begin{tikzcd}
&F(t+1)\ar[d]&F(t)\ar[d]&F(t-1)\ar[d]\\
\cdots\ar[r]&X(t+1)\ar[r]&X(t)\ar[r]&X(t-1)\ar[r]&\cdots
\end{tikzcd}\end{center}
be a tower of spectra, where $F(t) = \Fib(X(t)\rightarrow X(t-1))$. Then there is a spectral sequence
\begin{equation}\label{eq:sseq}
E_2^{s,t} = \pi_s F(t)\Rightarrow \pi_s \lim_{n\rightarrow\infty} X(n),\qquad d_r^{s,t}\colon E_r^{s,t}\rightarrow E_r^{s-1,t+r-1}.
\end{equation}
Write $Z_r^{s,t}$ and $B_r^{s,t}$ for the $r$-cycles and $r$-boundaries of this spectral sequence, so that
\[
0 = B_1^{s,t}\subset B_2^{s,t}\subset\cdots\subset Z_2^{s,t}\subset Z_1^{s,t} = \pi_s F(t),\qquad E_{r}^{s,t} = Z_{r-1}^{s,t}/B_{r-1}^{s,t}.
\]
We will find it convenient to interpret the differentials in this spectral sequence as relations, just as in \cite{bousfield1989homotopy}, so we begin by recalling the construction in this form. Define
\begin{equation}\label{eq:Dst}
D_r^{s,t} = \pi_s F(t)\times_{\pi_s X(t)}\im\left(\pi_s X(t+r-2)\rightarrow \pi_s X(t)\times \pi_{s-1}F(t+r-1)\right),
\end{equation}
where $\pi_s X(t+r-2)\rightarrow \pi_{s-1}F(t+r-1)$ is obtained from the boundary map $X(t+r-2)\rightarrow\Sigma F(t+r-1)$. Note that
\[
D_r^{s,t}\subset \pi_s F(t)\times \pi_{s-1}F(t+r-1) = Z_1^{s,t}\times Z_1^{s-1,t+r-1}.
\]
Recall the following basic fact about additive relations.

\begin{lemma}\label{lem:relation}
Let $M$ and $N$ be abelian groups and $R\subset M\times N$ a subgroup. Define
\[
Z = \im(R\rightarrow M) \qquad
K = \ker(R\rightarrow M) \qquad
B = \im(K\rightarrow N) \qquad
C = \coker(B\rightarrow N).
\]
Then the relation $\im(R\rightarrow Z\times C)$ gives a well defined function $Z\rightarrow C$.
\qed
\end{lemma}

The spectral sequence of \cref{eq:sseq} is now given as follows.

\begin{lemma}[Definition]\label{lem:towerss}
The following hold, where $Z_1^{s,t} = \pi_s F(t) = E_2^{s,t}$ as above.
\begin{enumerate}
\item $Z_{r-1}^{s,t} = \im(D_r^{s,t}\rightarrow Z_1^{s,t})$;
\item $B_{r-1}^{s-1,t+r-1} = \im(\ker(D_r^{s,t}\rightarrow Z_1^{s,t})\rightarrow Z_1^{s-1,t+r-1})$;
\item $d_r^{s,t}\colon Z_{r-1}^{s,t}\rightarrow Z_1^{s-1,t+r-1}/B_{r-1}^{s-1,t+r-1}$ is the function associated the relation $D_r^{s,t}$;
\item $Z_r^{s,t} = \ker(D_r^{s,t} \rightarrow Z_1^{s-1,t+r-1})$;
\item $B_r^{s-1,t+r-1} = \im(D_r^{s,t}\rightarrow Z_1^{s-1,t+r-1})$.
\qed
\end{enumerate}
\end{lemma}

\subsection{The homotopy limit spectral sequence}\label{ssec:shfpss}

Given a diagram $M\colon \calJ\rightarrow \Ab$ of abelian groups, let $\calH(\calJ;M)$ denote the limit of the composite
\begin{center}\begin{tikzcd}
\calJ\ar[r,"M"]&\Ab\ar[r,"H"]&\Sp
\end{tikzcd},\end{center}
and let 
\[
\calH^n(\calJ;M) = \Omega^{\infty-n}\calH(\calJ;M),\qquad H^n(\calJ;M) = \pi_{-n}\calH(\calJ;M) = \pi_0 \calH^n(\calJ;M).
\]
We may identify $H^n(\calJ;\bs)$ as the $n$th right derived functor of $\lim_{j\in\calJ}\colon\Fun(\calJ,\Ab)\rightarrow\Ab$. Given a diagram $X\colon \calJ\rightarrow\Sp$, each $\pi_t X$ is a diagram $\calJ\rightarrow\Ab$. The HLSS
\[
E_2^{s,t} = H^{t-s}(\calJ;\pi_t X)\Rightarrow \pi_s \lim_{j\in\calJ} X(j)
\]
is then the spectral sequence associated to the tower
\begin{center}\begin{tikzcd}
&\Sigma^{t+1}\calH(\calJ;\pi_{t+1} X)\ar[d]&\Sigma^t\calH(\calJ;\pi_t X)\ar[d]&\Sigma^{t-1}\calH(\calJ;\pi_{t-1}X)\ar[d]\\
\cdots\ar[r]&\lim_{j\in\calJ} (X(j)_{\leq t+1})\ar[r]&\lim_{j\in\calJ} (X(j)_{\leq t})\ar[r]&\lim_{j\in\calJ}(X(j)_{\leq t-1})\ar[r]&\cdots
\end{tikzcd}.\end{center}
Note in particular 
\[
D_r^{s,t} \subset H^{t-s}(\calJ;\pi_t X)\times H^{t+r-s}(\calJ;\pi_{t+r-1}X)
\]
and
\[
E_r^{s,t}=0\qquad\text{for}\qquad t<s.
\]

\subsection{Unstable homotopy limits}\label{ssec:unstable}

The preceding construction is not quite sufficient for our purposes, as it lacks the naturality properties we require. If $X,Y\colon \calJ\rightarrow\Sp$ are two diagrams of spectra, then a natural transformation $X\rightarrow Y$ does induce a map of HLSSs in the usual way; however, we are interested in the more exotic situation where we are given a natural transformation $\Omega^\infty X \rightarrow \Omega^\infty Y$ of diagrams of spaces, not necessarily stable. Here, one may suppose without loss of generality that $X$ and $Y$ are valued in connective spectra. 

All of our applications described in \cref{ssec:applications} are of this form, requiring an space-level analysis of unstable natural transformations. For example, \cref{thm:main} will follow from a consideration of the natural norm map
\[
\calP_\alpha\colon \Omega^\infty\Sp^K(S^\alpha,\res^G_KR)\rightarrow\Omega^\infty\Sp^G(S^{\Ind_K^G\alpha},R)
\]
of \cref{eq:spacenorm}, where $R$ is a diagram of $G\hyp\bbE_\infty$ rings. This map is pointed, but is essentially never stable. To access naturality with respect to this sort of map, we need a construction of the HLSS which depends on only the underlying diagram of spaces.

Let $T$ be a space. Then $T$ has a Postnikov tower:
\[
\cdots\rightarrow T_{\leq t+1}\rightarrow T_{\leq t}\rightarrow T_{\leq t-1}\rightarrow\cdots \rightarrow T_{\leq 1}\rightarrow T_{\leq 0} = \pi_0 T.
\]
The layers of this tower are determined by suitable $k$-invariants. If $T$ is simply connected then these are of the form $T_{\leq t-1}\rightarrow K(\pi_t T,t+1)$, but the situation is more subtle in general: if $T$ is merely pointed and connected then the target must take into account the natural action of $\pi_1 T$ on $\pi_t T$, and in the most general case one must instead consider a variant of $K(\pi_tT,t+1)$ which regards ``$\pi_t T$'' as a bundle of groups over the fundamental groupoid of $T$.

We are in a certain reasonably pleasant middle ground where $T$ need not be connected, but for all points $x\in T$ and $t\geq 1$ the natural action of $\pi_1(T,x)$ on $\pi_t(T,x)$ is trivial. Call such a space \textit{simple}. Informally, a simple space is a disjoint union of connected spaces for which the theory of Postnikov towers is at its simplest. For $n,m\geq 1$, let $B_{\pi_0 T}^m\Pi_n T = \coprod_{t\in \pi_0 T} K(\pi_n(T,t),m)$. Then we have the following standard fact.

\begin{lemma}\label{lem:simplepostnikov}
If $T$ is a simple space, then there are Cartesian squares
\begin{center}\begin{tikzcd}
K(\pi_t(T,x),t)\ar[r]\ar[d]&T_{\leq t}\ar[r]\ar[d,"p_t"]&\pi_0 T\ar[d]\\
\{p_t(x)\}\ar[r]&T_{\leq t-1}\ar[r]&B^{t+1}_{\pi_0 T}\Pi_t T
\end{tikzcd},\end{center}
where the right square always exists naturally in $T$, and the left square exists naturally in $T$ and the choice of a point $x\in T_{\leq t}$, provided such a point exists.
\qed
\end{lemma}

In other words, $B^{t+1}_{\pi_0 T}\Pi_t T$, treated as a bundle of abelian groups over the discrete space $\pi_0 T$, is the correct replacement for $K(\pi_tT,t+1)$ in the theory of Postnikov towers for non-connected simple spaces. It is, in particular, natural in $T$.  This bit of maneuvering would not be necessary if we restricted ourselves to considering only the case where $T$ is connected. In the context of the main theorem of this section, \cref{thm:unstablenaturality}, it is needed only to account for what happens with the path components living at the very fringe of the spectral sequence. 

Now say that $T\colon \calJ\rightarrow\Gpd_\infty$ is a diagram of simple spaces. Let
\[
H^0(\calJ;\pi_0 T) = \lim_{j\in\calJ}\pi_0 T(j).
\]
Observe that as $T$ is simple, if $x\in H^0(\calJ;\pi_0 T)$ and $t\geq 1$, then $\pi_t(T,x)$ is naturally a $\calJ$-shaped diagram of abelian groups. Let
\[
\calH^{t+1}_{\pi_0 T}(\calJ;\Pi_t T) = \lim_{j\in\calJ}B^{t+1}_{\pi_0 T(j)}\Pi_t T(j) \simeq \coprod_{x\in H^0(\calJ;\pi_0 T)}\calH^{t+1}(\calJ;\pi_t(T,x)).
\]

\begin{lemma}
There are Cartesian squares
\begin{center}\begin{tikzcd}
\calH^t(\calJ;\pi_t(T,x))\ar[r]\ar[d] & \lim_{j\in\calJ} (T(j)_{\leq t})\ar[r]\ar[d,"p_t"]&H^0(\calJ;\pi_0 T)\ar[d]\\
\{p_t(x)\}\ar[r]&\lim_{j\in\calJ} (T(j)_{\leq t-1})\ar[r]&\calH^{t+1}_{\pi_0 T}(\calJ;\Pi_t T)
\end{tikzcd},\end{center}
where the right square always exists naturally in $T$, and the left square exists naturally in $T$ and the choice of a point $x\in \lim_{j\in\calJ} (T(j)_{\leq t})$, provided such a point exists.
\end{lemma}
\begin{proof}
This follows by taking limits over \cref{lem:simplepostnikov}.
\end{proof}

Fix $s \geq 0$, $t\geq 0$, $r\geq 2$ and $x\in \pi_0\lim_{j\in\calJ}(T(j)_{\leq t+r-2})$, and write the same for the image of $x$ in $\pi_0\lim_{j\in\calJ}(T(j)_{\leq n})$ for $n\leq t+r-2$. Define
\begin{equation}\label{eq:Dun}
D_{r,x}^{s,t} = \lim\left(\pi_s p_t^{-1}(x)\rightarrow\pi_s \left(\lim_{j\in\calJ}(T(j)_{\leq t}),x\right)\leftarrow I^{s,t}_{r,x}\right),
\end{equation}
where
\[
I_{r,x}^{s,t} = \im\left(\pi_s\left(\lim_{j\in\calJ}(T(j)_{\leq t+r-2}),x\right)\rightarrow \pi_s \left(\lim_{j\in\calJ}(T(j)_{\leq t}),x\right)\times \pi_s\left(\calH^{t+r}_{\pi_0 T}(\calJ;\Pi_{t+r-1}X),x\right)\right).
\]
When $s = 0$, we extend this notation to be defined having fixed just $x\in \pi_0\lim_{j\in\calJ}(T(j)_{\leq t-1})$. We will only make use of the simplest case, where $T$ is pointed and either $s=t=0$ or $x$ is the basepoint, but make no such restriction for the moment. Observe that $D_{r,x}^{s,t}\subset J_{r,x}^{s,t}$ where
\begin{equation}\label{eq:J}
J_{r,x}^{s,t} = \begin{cases} H^{t-s}(\calJ;\pi_t(T,x))\times H^{t+r-s}(\calJ;\pi_{t+r-1}(T,x)), & s\geq 1; \\
H^t(\calJ;\pi_t(T,x))\times \coprod_{y\in H^0(\calJ;\pi_0 T)} H^{t+r}(\calJ;\pi_{t+r-1}(T,y)),&s=0.
\end{cases}
\end{equation}

Let $S\colon\calJ\rightarrow\Gpd_\infty$ be another diagram of simple spaces, and $f\colon T\rightarrow S$ a map of diagrams. This induces maps
\[
f\colon \lim_{j\in\calJ}(T(j)_{\leq t})\rightarrow\lim_{j\in\calJ}(S(j)_{\leq t})
\]
of spaces, and for $x\in H^0(\calJ;\pi_0 T)$ and $t\geq 1$, a map
\[
f\colon \pi_t(T,x)\rightarrow \pi_t(S,f(x))
\]
of diagrams of abelian groups. Combined, these yield
\[
f\colon J_{r,x}^{s,t}(T)\rightarrow J_{r,f(x)}^{s,t}(S).
\]

\begin{lemma}\label{lem:naturaldiff}
The map $f\colon J_{r,x}^{s,t}(T)\rightarrow J_{r,f(x)}^{s,t}(S)$ satisfies $f(D_{r,x}^{s,t}(T))\subset D_{r,f(x)}^{s,t}(S)$.
\end{lemma}
\begin{proof}
This is clear from the construction.
\end{proof}

\subsection{Comparing the stable and unstable constructions}\label{ssec:stableunstable}

Let $X$ be a spectrum, and consider the underlying simple space $\Omega^\infty X$. For $x\in \pi_0 X$, write $\Omega^\infty_x X$ for the path component of $\Omega^\infty X$ corresponding to $x$. As $\Omega^\infty X$ is a group, there are equivalences
\[
\gamma_x\colon \Omega^\infty_x X\rightarrow\Omega^\infty_0 X,\qquad \gamma_x(a) = a-x
\]

\begin{lemma}\label{lem:groupsplit}
The above patch together into an equivalence
\[
\Omega^\infty X\simeq\pi_0 X\times \Omega^\infty_0 X,
\]
compatible on Postnikov towers with equivalences
\[
B^{t+1}_{\pi_0 X}\Pi_t X \cong \pi_0 X \times K(\pi_t X,t+1)
\]
for $t\geq 1$. These equivalences are natural with respect to $\Omega^\infty X$ as a group object.
\qed
\end{lemma}

Now say that $X$ is a diagram of spectra, and consider the underlying diagram $\Omega^\infty X$ of simple spaces. For $s\geq 0$ and $r\geq 2$, define 
\begin{equation}\label{eq:Js}
J^{s,t}_{r} = \begin{cases}
H^{t-s}(\calJ;\pi_t X) \times H^{t+r-s}(\calJ;\pi_{t+r-1}X),&s\geq 1 \\
H^t(\calJ;\pi_t X)\times H^0(\calJ;\pi_0 X)\times H^{t+r}(\calJ;\pi_{t+r-1}X),&s=0.\\
\end{cases}
\end{equation}

\begin{lemma}\label{lem:J}
Let $D_{r,x}^{s,t}$ and $J_{r,x}^{s,t}$ and be defined as in \cref{eq:Dun} and \cref{eq:J} for the diagram $\Omega^\infty X$. Then there are isomorphisms
\[
J_{r,x}^{s,t}\cong J_r^{s,t},
\]
and 
\[
D_{r,x}^{s,t}=\begin{cases}
D_{r,0}^{s,t} & s \geq 1,\\
D_{r,0}^{0,t} & s = 0 \text{ and we have a lift of $x$ to }\pi_0 \lim_{j\in\calJ}(X_{\leq t+r-2}),\\
\emptyset & \text{otherwise,}
\end{cases}
\]
as subsets of $J_r^{s,t}$. These identifications are natural in $\Omega^\infty X$ as a diagram of group objects.
\end{lemma}
\begin{proof}
As $X$ is a diagram of spectra, $\Omega^\infty X$ is a diagram of group objects. The lemma then follows by applying \cref{lem:groupsplit} to the constructions of the sets involved.
\end{proof}

There are obvious maps
\begin{equation}\label{eq:q}
q\colon J^{s,t}_r\rightarrow H^{t-s}(\calJ;\pi_t X)\times H^{t+r-s}(\calJ;\pi_{t+r-1}) = Z_1^{s,t}\times Z_1^{s-1,t+r-1},
\end{equation}
given by the identity for $s\geq 1$ and the projection $q(w,x,y) = (w,y)$ for $s=0$.

\begin{lemma}\label{lem:unstablediff}
Recall $D_r^{s,t}\subset H^{t-s}(\calJ;\pi_t X)\times H^{t+r-s}(\calJ;\pi_{t+r-1}X)$ and $D_{r,0}^{s,t}\subset J_{r,0}^{s,t}\cong J_r^{s,t}$ from \cref{eq:Dst} and \cref{eq:Dun}. We have
\[
D_{r}^{s,t} = \im\left(q\colon D_{r,0}^{s,t}\rightarrow H^{t-s}(\calJ;\pi_t X)\times H^{t+r-s}(\calJ;\pi_{t+r-1} X)\right)
\]
for $s\geq 0$. Moreover,
\[
q^{-1}(D_r^{s,t}) = \begin{cases}
D_{r,0}^{s,t}, & s\geq 1; \\
\{(x,0,y):(x,y)\in D_r^{0,t}\},&s=0,\,t\geq 1; \\
\{(x,x,y):(x,y)\in D_r^{0,0}\},&s=t=0.
\end{cases}
\]
\end{lemma}
\begin{proof}
Immediate from the definitions.
\end{proof}

The following now relates the stable construction of \cref{lem:towerss} with the above unstable constructions.

\begin{lemma}\label{lem:towerssun}
The HLSS for $X$ satisfies the following for $s \geq 0$.
\begin{enumerate}
\item $Z_{r-1}^{s,t} = \im(D_{r,0}^{s,t}\rightarrow H^{t-s}(\calJ;\pi_t X))$;
\item $Z_r^{s,t} = \im(D_{r,0}^{s,t}\times_{H^{t+r}(\calJ;\pi_{t+r-1}X)}\{0\}\rightarrow H^t(\calJ;\pi_t X))$;
\item $B_{r-1}^{s-1,t+r-1} = \im(\{0\}\times_{H^{t-s}(\calJ;\pi_t X)}D_{r,0}^{s,t}\rightarrow H^{t+r-s}(\calJ;\pi_{t+r-1}X))$;
\item $B_r^{s-1,t+r-1} = \im(D_{r,0}^{s,t}\rightarrow H^{t+r-s}(\calJ;\pi_{t+r-1}X))$;
\item For $x\in E_r^{s,t}$ and $y\in E_r^{s-1,t+r-1}$, we have $d_r(x) = y$ if and only if $x$ and $y$ lift to elements of $H^{t-s}(\calJ;\pi_t X)$ and $H^{t+r-s}(\calJ;\pi_{t+r-1}X)$ respectively with the property that
\begin{enumerate}
\item If $s\geq 1$, then $(x,y)\in D_{r,0}^{s,t}$;
\item If $s=0$ and $t\geq 1$, then $(x,0,y)\in D_{r,0}^{0,t}$;
\item If $s=t=0$, then $(x,x,y)\in D_{r,0}^{0,0}$.
\end{enumerate}
\end{enumerate}
\end{lemma}
\begin{proof}
These follow from \cref{lem:towerss} and \cref{lem:unstablediff}.
\end{proof}

\subsection{Naturality with respect to pointed maps}\label{ssec:unstablenaturality}

Let $X$ and $Y$ be spectra, and let
\[
F\colon \Omega^\infty X\rightarrow\Omega^\infty Y
\]
be a map of pointed spaces. For $n\geq 0$, let
\begin{equation}\label{eq:Q}
Q\colon \pi_n X\rightarrow \pi_n Y
\end{equation}
be the map induced by $\pi_n(\bs,0)$. For $x\in \pi_0 X$ and $n\geq 1$, write
\begin{equation}\label{eq:Qx}
Q_x = \gamma_{Q(x)}\circ \pi_n(\bs,x)\circ \gamma_x^{-1}\colon \pi_n X\cong \pi_n(X,x)\rightarrow \pi_n(Y,Q(x))\cong \pi_n Y.
\end{equation}
In particular, $Q_0 = Q$.

\begin{lemma}\label{lem:shiftnat}
Define
\[
Q_\bs\colon \pi_0 X \times K(\pi_t X,t+1)\rightarrow \pi_0 Y\times K(\pi_t Y,t+1),\qquad Q_\bs(x,y) = (Q(x),Q_x(y)).
\]
Then the diagram
\begin{center}\begin{tikzcd}
B^{t+1}_{\pi_0 X}\Pi_t X\ar[r,"\simeq"]\ar[d,"F"]&\pi_0 X\times K(\pi_t X,t+1)\ar[d,"Q_\bs"]\\
B^{t+1}_{\pi_0 Y}\Pi_t Y\ar[r,"\simeq"]&\pi_0 Y\times K(\pi_t Y,t+1)
\end{tikzcd}\end{center}
commutes, where the left vertical map is that naturally induced from the map $F\colon \Omega^\infty X\rightarrow\Omega^\infty Y$ of simple spaces.
\end{lemma}
\begin{proof}
This holds by construction.
\end{proof}

Now suppose that $X,Y\colon\calJ\rightarrow\Sp$ are diagrams of spectra, and fix a map
\[
F\colon \Omega^\infty X\rightarrow\Omega^\infty Y
\]
of diagrams of pointed spaces. As before, write
\[
Q\colon \pi_0 X\rightarrow \pi_0 Y
\]
for the map on path components. Following \cref{lem:naturaldiff} and discussion it succeeds, there are maps
\[
F\colon J^{s,t}_{r,x}(\Omega^\infty X)\rightarrow J^{s,t}_{r,Q(x)}(\Omega^\infty Y),
\]
and these satisfy
\[
F(D_{r,x}^{s,t})\subset D_{r,Q(x)}^{s,t}
\]
for $s \geq 0$. On the other hand, because $F$ is a map of diagrams of pointed simple spaces, there are maps
\[
Q_x\colon \pi_t X \rightarrow \pi_t Y
\]
of diagrams of groups for $x\in H^0(\calJ;\pi_0 X)$ and $t\geq 1$, induced by \cref{eq:Qx}. We abbreviate $Q_0$ to $Q$. These induce maps on $H^\ast(\calJ;\bs)$.

\begin{lemma}\label{lem:shiftdiff}
Recall the sets $J_r^{s,t}$ from \cref{eq:Js}. Define
\[
Q_+\colon J_r^{s,t}(X)\rightarrow J_r^{s,t}(Y)
\]
for $s \geq 0$ by
\[
\begin{cases}
Q_+(w,y) = (Q(w),Q(y)), & s\geq 1; \\
Q_+(w,x,y) = (Q(w),Q(x),Q(y)), & s = 0,\, t\geq 1; \\
Q_+(w,x,y) = (Q(w),Q(x),Q_x(y)), & s=t=0.
\end{cases}
\]
Then $Q_+ = F$, in the sense that the diagram
\begin{center}\begin{tikzcd}
D_{r,0}^{s,t}(\Omega^\infty X)\ar[r,"F"]\ar[d,"\subset"]&D_{r,0}^{s,t}(\Omega^\infty Y)\ar[d,"\subset"]\\
J_r^{s,t}(X)\ar[r,"Q_+"]&J_r^{s,t}(Y)
\end{tikzcd}\end{center}
commutes, where the top horizontal map is as in \cref{lem:naturaldiff}.
\end{lemma}
\begin{proof}
By \cref{lem:naturaldiff}, this diagram commutes should we replace the bottom map with $F\colon J_{r,0}^{s,t}(X)\rightarrow J_{r,0}^{s,t}(Y)$. By \cref{lem:J}, $Q_+$ and $F$ have isomorphic domain and codomain, and we must only check that $Q_+ = F$ under this isomorphism. For $t\geq 1$, the map $F$ is just that induced by the pointed map $F\colon \Omega^\infty X\rightarrow\Omega^\infty Y$ on homotopy groups at the basepoint, which is exactly as described by $Q_+$. Now consider $s=t=0$. Define
\[
Q_\bs\colon H^0(\calJ;\pi_0 X)\times \calH^{r+1}(\calJ;\pi_r X)\rightarrow H^0(\calJ;\pi_0 Y)\times \calH^{r+1}(\calJ;\pi_r Y),~~ Q_\bs(x,y) = (Q(x),Q_x(y)).
\]
Taking limits over \cref{lem:shiftnat}, we find that the diagram
\begin{center}\begin{tikzcd}
H^0(\calJ;\pi_0 X)\times \calH^{r+1}(\calJ;\pi_r X)\ar[r,"Q_\bs"]&H^0(\calJ;\pi_0 Y)\times \calH^{r+1}(\calJ;\pi_r Y) \\
\lim_{j\in\calJ}\Omega^\infty(X_{\leq r-1})\ar[r]\ar[d]\ar[u]&\lim_{j\in\calJ}\Omega^\infty(Y_{\leq r-1})\ar[d]\ar[u]\\
H^0(\calJ;\pi_0 X)\ar[r,"Q"]&H^0(\calJ;\pi_0 Y)
\end{tikzcd}\end{center}
commutes. By definition, $F = Q\times Q_\bs$ as maps $J_r^{0,0}(X)\rightarrow J_r^{0,0}(Y)$, and this is exactly $Q_+$ as described.
\end{proof}

We can now give the main naturality theorem.

\begin{theorem}\label{thm:unstablenaturality}
Given diagrams $X,Y\colon \calJ\rightarrow\Sp$ of spectra and map $F\colon \Omega^\infty X\rightarrow\Omega^\infty Y$ of diagrams of pointed spaces, the maps $Q$ and $Q_x$ of \cref{eq:Q} and \cref{eq:Qx} interact with the HLSSs for $\lim_{j\in\calJ}X(j)$ and $\lim_{j\in\calJ}Y(j)$ as follows.
\begin{enumerate}
\item $Q(Z_r^{s,t}(X))\subset Z_r^{s,t}(Y)$ for $s \geq 0$;
\item $Q(B_r^{s,t}(X))\subset B_r^{s,t}(Y)$ for $s \geq -1$;
\item If $x\in Z_r^{0,0}(X)$ then $Q_x(B_{r-1}^{-1,r-1}(X))\subset B_{r-1}^{-1,r-1}(Y)$;
\item For $x\in E_r^{s,t}(X)$ with $s\geq 0$, we have
\[
d_r(Q(x)) = \begin{cases} Q(d_r(x)),& t\geq 1,\\
Q_x(d_r(x)),& t=s=0;
\end{cases}
\]
\item For $s \geq 0$, if $x\in E_2^{s,t}(X)$ is a permanent cycle detecting $f\in \pi_s \lim_{j\in\calJ}X(j)$, then the permanent cycle $Q(x)\in E_2^{s,t}(Y)$ detects $Q(f)$ modulo classes in higher filtration.
\end{enumerate}
\end{theorem}
\begin{proof}
(1)--(4) follow from \cref{lem:shiftdiff}, which describes the map $D_{r,0}^{s,t}(\Omega^\infty X)\rightarrow D_{r,0}^{s,t}(\Omega^\infty Y)$ induced by $F$, and \cref{lem:towerssun}, which explains how cycles, boundaries, and differentials are naturally defined in terms of $D_{r,0}^{s,t}$. Let us just illustrate this with a proof of (3).

Let $x\in Z_r^{0,0}(X)$ and $y\in B_{r-1}^{-1,r-1}(X)$. This implies $(x,x,y)\in D_{r,0}^{0,0}(\Omega^\infty X)$, and thus $(Q(x),Q(x),Q_x(y)) \in D_{r,0}^{0,0}(\Omega^\infty Y)$. Taking $y = 0$ shows $Q(x) \in Z_r^{0,0}(Y)$. As $Q(x) \in Z_r^{0,0}(Y)$ and $(Q(x),Q(x),Q_x(y)) \in D_r^{0,0}(Y)$, it follows that $Q_x(y)\in B_{r-1}^{-1,r-1}(Y)$ as claimed.

(5) holds as $Q$ is compatible with the maps $\lim_{j\in\calJ}\Omega^\infty(X_{\leq t})\rightarrow \lim_{j\in\calJ}\Omega^\infty(Y_{\leq t})$.
\end{proof}

\section{Looping power operations}\label{sec:pow}

If $\calC$ is a stable category, then for any $X,Y\in\calC$ one may form the mapping spectrum $\calC(X,Y)$. This construction preserves limits in $Y$, allowing one to form HLSSs for diagrams in arbitrary stable categories. If $N\colon \calC\rightarrow\calD$ is a pointed functor between stable categories, then for any $X,Y\in\calC$ one obtains a map
\[
 \Omega^\infty\calC(X,Y) = \Map_\calC(X,Y)\rightarrow\Map_\calD(NX,NY) = \Omega^\infty\calD(NX,NY)
\]
of pointed spaces. \cref{thm:unstablenaturality} describes how these maps appear in HLSSs, at least once one understands how they behave on higher homotopy groups. This section describes explicitly what happens in the main example of interest, eventually leading to \cref{thm:poweroff}. In \cref{ssec:powhlss}, we put everything together, yielding \cref{thm:main} and variations thereon.

\subsection{Looping power operations}

Fix a compact Lie group $G$, let $\Sp^G$ be the category of $G$-spectra, and for $m\geq 0$ write
\[
\bbP^m\colon \Sp^G\rightarrow\Sp^{\Sigma_m\wr G}
\]
for the $m$-fold smash power functor. These are the functors denoted $\wedge^m$ in \cite{bohmann2014comparison}. Note that the group $G$ will not play a real role in the following. Write $\rho_m$ for the permutation representation of $\Sigma_m$ on $\bbR^m$, and observe that
\[
\bbP^m(S^\alpha) \simeq S^{\rho_m\otimes\alpha}
\]
for $\alpha\in RO(G)$. Thus, external power operations in this context take the form
\[
P^m_\alpha\colon \pi_\alpha^G X \rightarrow \pi_{\rho_m\otimes\alpha}^{\Sigma_m\wr G}\bbP^m X
\]
for $X\in\Sp^G$. Given $x\in \pi_\alpha X$ and $n\geq 1$, write
\[
P^{m,(n)}_{\alpha,x}\colon \pi_{n+\alpha}^G X\rightarrow \pi_{n+\rho_m\otimes \alpha}^{\Sigma_m\wr G}\bbP^m X
\]
for the composite
\[
\pi_{n+\alpha} X \cong \pi_n (\Map_{\Sp^G}(S^\alpha,X),x)\rightarrow \pi_n\Map_{\Sp^{\Sigma_m\wr G}}(S^{\rho_m\otimes\alpha},\bbP^m X),P^m_\alpha(x))\cong \pi_{n+\rho_m\otimes\alpha}^{\Sigma_m\wr G} \bbP^m X,
\]
the inner map being induced by functoriality of $\bbP^m$. The goal of this subsection is to describe the operations $P^{m,(n)}_{\alpha,x}$ explicitly. This description is given in \cref{thm:poweroff}, the proof of which amounts to a collection of standard observations about the behavior of the functors $\bbP^m$, which we now make.

\subsubsection{Euler classes}

Write $\ol{\rho}_m$ for the reduced permutation representation of $\Sigma_m$. This may be regarded as a representation of $\Sigma_m\wr G$ by restriction along the projection $\Sigma_m\wr G\rightarrow\Sigma_m$. Write $a_m\in \pi_{-\ol{\rho}_m}^{\Sigma_m\wr G} S_{\Sigma_m\wr G}$ for the Euler class of $\ol{\rho}_m$, i.e.\ the class represented by the inclusion of poles $S^0\rightarrow S^{\ol{\rho}_m}$, or what is equivalent, the inclusion of fixed points $S^1 \rightarrow S^{\rho_m}$.

\subsubsection{The addition formula}\label{ssec:addition}

The functors $\bbP^m$ satisfy
\[
\bbP^m(A\oplus B)\simeq \bigoplus_{i+j=m}\Ind^{\Sigma_m\wr G}_{\Sigma_{i,j}\wr G}\left(\bbP^i(A)\boxtimes \bbP^j(B)\right),
\]
where $\Sigma_{i,j} = \Sigma_i\times \Sigma_j\subset \Sigma_{i+j}$  and $\bbP^i(A)\boxtimes\bbP^j(B)$ is $\bbP^i(A)\otimes\bbP^j(B)$ considered with its natural $\Sigma_{i,j}\wr G$-action. This allows us to identify the operation
\[
P_{(\alpha,\beta)}^m\colon \pi_\alpha^G X \times \pi_\beta^G X = [S^{\alpha}\oplus S^{\beta},X]\rightarrow [\bbP^m(S^\alpha\oplus S^\beta),\bbP^m X]
\]
as
\[
P_{(\alpha,\beta)}^m(x,y) = \sum_{i+j=m}\tr_{\Sigma_{i,j}\wr G}^{\Sigma_m\wr G}(P_\alpha^i(x)\cdot P_\beta^j(y)).
\]
Here, the products appearing on the right are external products of signature
\[
\pi_{\rho_i\otimes\alpha}^{\Sigma_i\wr G}\bbP^i X \otimes \pi_{\rho_j\otimes\beta}^{\Sigma_j\wr G}\bbP^j X \rightarrow
 \pi_{\rho_i\otimes\alpha+\rho_j\otimes\beta}^{\Sigma_i\wr G \times \Sigma_j\wr G}\left(\bbP^i X \boxtimes \bbP^j X\right)
  = \pi_{\rho_i\otimes\alpha+\rho_j\otimes\beta}^{\Sigma_{i,j}\wr G} \res^{\Sigma_m\wr G}_{\Sigma_{i,j}\wr G}\bbP^m X.
\]

\subsubsection{Colimit comparison maps}\label{ssec:looping}

For a space $F$ and object $A$, write $F\cdot A = \colim_{x\in F}A$ for the unbased tensor with $F$. The basepoint of $S^n$ yields a natural retraction
\[
A\rightarrow S^n\cdot A\rightarrow A,
\]
and this gives rise to a splitting
\begin{equation}\label{eq:snsplit}
S^n\cdot A\simeq \Sigma^nA \oplus A.
\end{equation}
Observe that there are natural colimit comparison maps
\[
S^n\cdot \bbP^m(A)\rightarrow \bbP^m(S^n\cdot A).
\]
As $\bbP^m$ is compatible with the monoidal structure, these are determined by their effect when $A = S^0$, i.e.\ by the map
\begin{equation}\label{eq:sncolimt}
S^n\oplus S^0 \simeq S^n\cdot \bbP^m(S^0) \rightarrow \bbP^m(S^n\cdot S^0) \simeq\bigoplus_{i+j=m}\Ind_{\Sigma_{i,j}\wr G}^{\Sigma_m\wr G}S^{\rho_i\otimes n}.
\end{equation}
This is the map given by the unreduced suspension spectrum of the diagonal
\[
S^n \rightarrow (S^n)^{\times m}
\]
map of spaces. The splitting of the target in \cref{eq:sncolimt} as a direct sum amounts to the standard splitting $\Sigma(X_1\times\cdots\times X_m)\simeq \Sigma \bigvee_{I\subset\{1,\ldots,m\}}\bigwedge_{i\in I}X_i$, valid for pointed spaces $X_1,\ldots,X_m$. The restriction of \cref{eq:sncolimt} to $S^0$ is just the inclusion into the $i = 0$ summand. On $S^n$, one has maps
\[
S^n\rightarrow \Ind_{\Sigma_{i,j}\wr G}^{\Sigma_m\wr G}S^{\rho_i \otimes n},
\]
which are seen to be adjoint to the inclusion of fixed points $a_i^n\colon S^n\rightarrow S^{\rho_i\otimes n}$.

\subsubsection{Looping operations}

Abbreviate $\Map_G = \Map_{\Sp^G}$, and consider the diagram
\begin{center}\begin{tikzcd}[column sep=1mm]
\Map_{G}(S^{n+\alpha},X)\ar[r]\ar[d]&\Map_{\Sigma_m\wr G}(\bbP^m(S^n\cdot S^\alpha)/\bbP^m S^\alpha,\bbP^m X)\ar[r]\ar[d]&\Map_{\Sigma_m\wr G}(\Sigma^n \bbP^m(S^\alpha),\bbP^m X)\ar[d]\\
\Map_{G}(S^n\cdot S^\alpha,X)\ar[r]\ar[d,"q"]&\Map_{\Sigma_m\wr G}(\bbP^m(S^n\cdot S^\alpha),\bbP^m X)\ar[r]\ar[d]&\Map_{\Sigma_m\wr G}(S^n\cdot\bbP^m(S^\alpha),\bbP^m X)\ar[d,"p"]\\
\Map_{G}(S^\alpha,X)\ar[r]&\Map_{\Sigma_m\wr G}(\bbP^m(S^\alpha),\bbP^m X)\ar[r]&\Map_{\Sigma_m\wr G}(\bbP^m(S^\alpha),\bbP^m X)
\end{tikzcd}\end{center}
of spaces. Here, the bottom vertical maps are induced by the basepoint of $S^n$, and the columns are fiber sequences. By definition, $P_{\alpha,x}^{m,(n)}$ is the induced map
\[
\pi_0 q^{-1}(x)\rightarrow \pi_0 p^{-1}(P^m_\alpha(x)).
\]
Here, under the splitting of \cref{eq:snsplit}, we may write the inner row as
\begin{align}\label{eq:comp}
\begin{split}
\Map_G(S^{n+\alpha},X)&\times\Map_G(S^,X)\simeq\Map_G(S^{n+\alpha}\oplus S^\alpha,X)\\
&\rightarrow \Map_{\Sigma_m\wr G}(\bbP^m(S^{n+\alpha}\oplus S^\alpha),\bbP^m X)\\
&\rightarrow \Map_{\Sigma_m\wr G}(S^{n+\rho_m\otimes\alpha}\oplus S^{\rho_m\otimes\alpha},\bbP^m X)\\
&\simeq \Map_{\Sigma_m\wr G}(S^{n+\rho_m\otimes\alpha},\bbP^mX)\times\Map_{\Sigma_m\wr G}(S^{\rho_m\otimes\alpha},\bbP^mX),
\end{split}
\end{align}
and identify
\[
q^{-1}(x) = \Map_G(S^{n+\alpha},X)\times\{x\},\qquad p^{-1}(P^m_\alpha(x)) = \Map_{\Sigma_m\wr G}(S^{n+\rho_m\otimes\alpha},X)\times\{P^m_\alpha(x)\}.
\]
Putting this together for all $x$, on path components the composite \cref{eq:comp} yields the map
\begin{gather*}
P_{\alpha,\bullet}^{m,(n)}\colon \pi_{n+\alpha}^GX\times\pi_\alpha^G X\rightarrow \pi_{n+\rho_m\otimes\alpha}^{\Sigma_m\wr G} \bbP^m X \times \pi_{\rho_m\otimes\alpha}^{\Sigma_m\wr G}\bbP^m X,\\
 P_{\alpha,\bullet}^{m,(n)}(f,x) = (P_{\alpha,x}^{m,(n)}(f),P_\alpha^m(x)).
\end{gather*}

\subsubsection{Putting everything together}
\begin{theorem}\label{thm:poweroff}
Fix $X\in\Sp^G$, $\alpha\in RO(G)$, and $x\in \pi_\alpha X$. Then the operation
\[
P^{m,(n)}_{\alpha,x}\colon \pi_{n+\alpha}^G X\rightarrow \pi_{n+\rho_m\otimes \alpha}^{\Sigma_m\wr G}\bbP^m X
\]
is given by
\[
P^{m,(n)}_{\alpha,x}(f) = \sum_{0 < i \leq m} \tr_{\Sigma_{i,m-i}\wr G}^{\Sigma_m\wr G}\left(a_i^nP^i_{n+\alpha}(f)\cdot P^{m-i}_\alpha(x)\right)
\]
\end{theorem}
\begin{proof}
The first map in \cref{eq:comp} is described in \cref{ssec:addition}, and the second map is described in \cref{ssec:looping}. Tracing through these descriptions and identifying $P^{m,(n)}_{\alpha,x}(f)$ as the first coordinate of $P_{\alpha,\bullet}^{m,(n)}(f,x)$ yields the theorem.
\end{proof}

\subsection{Power operations in the HLSS}\label{ssec:powhlss}

Let $R\colon\calJ\rightarrow\Sp^G$ be a diagram of $G$-spectra. For each $\alpha\in RO(G)$, one may take mapping spectra levelwise to obtain a diagram $\Sp^G(S^\alpha,R)$ of spectra, with $\lim_{j\in\calJ}\Sp^G(S^\alpha,R(j))\simeq\Sp^G(S^\alpha,\lim_{j\in\calJ}R(j))$. Thus there is an HLSS
\[
E_2^{s+\alpha,t+\alpha} = H^{t-s}(\calJ;\pi_{t+\alpha}^G R)\Rightarrow\pi_{s+\alpha}^G\lim_{j\in\calJ}R(j).
\]
The composite $\bbP^mR\colon \calJ\rightarrow\Sp^{\Sigma_m\wr G}$ is likewise a diagram of $\Sigma_m\wr G$-spectra for each $m$, with its own HLSS $E_{\ast,m}^{\ast,\ast}$.
For each $\alpha\in RO(G)$ and $m\geq 1$, there are maps
\[
\calP^m_\alpha\colon\Map_{\Sp^G}(S^\alpha,R)\rightarrow \Map_{\Sp^{\Sigma_m\wr G}}(S^{\rho_m\otimes\alpha},\bbP^m R)
\]
of diagrams of pointed spaces. The extent to which this induces a map $E_\ast^{\ast,\ast}\rightarrow E_{\ast,m}^{\ast,\ast}$ of spectral sequences is exactly as described in \cref{thm:unstablenaturality}, once one understands how $\calP^m_\alpha$ behaves on higher homotopy groups, which is then as described in \cref{thm:poweroff}. Putting everything together, we learn the following.

As before, write $a_m$ for the class induced by the inclusion $S^1\rightarrow S^{\rho_m}$ of fixed points, where $\rho_m$ is the permutation representation of $\Sigma_m$.

\begin{theorem}\label{thm:powhfpss}
With notation as above,
\begin{enumerate}
\item The composite $a_m P_\alpha^m$ is additive. In particular, for $s\geq 0$ there are maps
\[
Q_\alpha^m = a_m^tP_{t+\alpha}^m\colon E_2^{s+\alpha,t+\alpha} \rightarrow E_{2,m}^{s+\rho_m\otimes\alpha,t+\rho_m\otimes\alpha}.
\]
\item $Q_\alpha^m(Z_r^{s+\alpha,t+\alpha})\subset Z_{r,m}^{s+\rho_m\otimes\alpha,t+\rho_m\otimes\alpha}$ for $s\geq 0$;
\item $Q_\alpha^m(B_r^{s+\alpha,t+\alpha})\subset B_{r,m}^{s+\rho_m\otimes\alpha,t+\rho_m\otimes\alpha}$ for $s\geq -1$;
\item For $x\in E_r^{s+\alpha,t+\alpha}$ with $s\geq 0$, we have
\[
d_r(Q_\alpha^m(x)) = \begin{cases}
Q_\alpha^m(d_r(x)),&t\geq 1;\\
Q_\alpha^m(d_r(x)) + \sum_{0 < i < m} \tr_{\Sigma_{i,j}\wr G}^{\Sigma_m\wr G}\left(Q_\alpha^i(d_r(x))\cdot Q_\alpha^{m-i}(x)\right),&s=t=0.
\end{cases}
\]
\item If $x\in E_2^{\alpha,t+\alpha}$ is a permanent cycle detecting $f\in \pi_\alpha^G \lim_{j\in\calJ}R(j)$, then $Q_\alpha^m(x)$ detects $P_\alpha^m(f) \in \pi_{\rho_m\otimes\alpha}^{\Sigma_m\wr G}\lim_{j\in\calJ}\bbP^m R(j)$ modulo classes in higher filtration.
\qed
\end{enumerate}
\end{theorem}

We now describe three specializations of \cref{thm:powhfpss}. The first is \cref{thm:main}. Let $G$ be a finite group and $R$ a $G\hyp\bbE_\infty$ ring. Let $K\subset G$ be a subgroup of index $m$ and $\alpha\in RO(K)$, and consider the norm
\[
P_\alpha\colon \pi^K_\alpha R \rightarrow \pi_{\Ind_K^G\alpha}^G R.
\]
This factors as
\begin{center}\begin{tikzcd}[column sep=small]
\pi^K_\alpha R \ar[r,"P^m_\alpha"]&\pi_{\rho_m\otimes\alpha}^{\Sigma_m\wr K}\bbP^m R\ar[r,"\res"]&\pi_{\res^{\Sigma_m\wr K}_G(\rho_m\otimes\alpha)}^G \res^{\Sigma_m\wr K}_G \bbP^m R\ar[r,"\cong"]&\pi_{\Ind_K^G\alpha}^G N_K^G R\ar[r,"N_K^G"]&\pi_{\Ind_K^G\alpha}^GR
\end{tikzcd}.\end{center}
Here, the restriction is along a suitable embedding $G\subset \Sigma_m\wr K$, and the final map is induced by the $G\hyp\bbE_\infty$ ring structure on $R$. The behavior of the first map with respect to HLSSs is what was described in \cref{thm:powhfpss}. The remaining maps are stable, and entirely compatible with HLSSs. Thus we may regard \cref{thm:main} as a specialization of \cref{thm:powhfpss}.

For the second, let $R$ be an ordinary $\bbE_\infty$ ring. Then for $n\in\bbZ$ and $m\geq 1$, the $m$th total power operation
\[
P^m\colon \pi_n R = [S^n,R]\rightarrow [(S^n)^{\otimes m}_{\h \Sigma_m},R]
\]
is the map induced on path components by the composite
\begin{center}\begin{tikzcd}[row sep=0mm]
\Map_{\Sp}(S^n,R)\ar[r,"P^m"]&\Map_{\Sp^{\Sigma_m}}(S^{\rho_m\otimes n},\bbP^m R)\\
\ar[r,"\colim"]&\Map_{\Sp}((S^n)^{\otimes m}_{\h \Sigma_m},R^{\otimes m}_{\h \Sigma_m})\ar[r,"\mu"]&\Map_{\Sp}((S^n)^{\otimes m}_{\h \Sigma_m},R)
\end{tikzcd}.\end{center}
The situation is analogous to the $G\hyp\bbE_\infty$ case.

For the third, let $R$ be an ultracommutative ring spectrum in the sense of \cite[Definition 5.1.1]{schwede2018global}. Then $R$ is equipped with strictly $\Sigma_m$-equivariant maps $R^{\otimes m}\rightarrow R$ within the category of orthogonal spectra. Following \cite[Theorem 4.5.25]{schwede2018global}, we may produce from $R$ the $G$-spectrum $U_G R$ by considering the orthogonal spectrum $R$ as an orthogonal $G$-spectrum with trivial $G$-action. The $\Sigma_m$-equivariant maps $R^{\otimes m}\rightarrow R$ may then be considered as maps $\bbP^m U_G R \rightarrow U_{\Sigma_m\wr G} R$, and this yields norms
\[
P^m_\alpha\colon \pi_\alpha^G R\rightarrow \pi_{\rho_m\otimes\alpha}^{\Sigma_m\wr G} R
\]
for $\alpha\in RO(G)$. The situation is now analogous to the previous examples. See \cite[Chapter 5]{schwede2018global} \cite{stahlhauer2021ginfty} for more on power operations in the global equivariant context, as well as \cite{greenleesmay1997localization} for earlier related material.

\section{Norms in the \texorpdfstring{$C_2$}{C\_2}-equivariant Adams spectral sequence}\label{sec:c2}

We now prove \cref{thm:c2}.
As $H\bbF_2^{C_2}$ is a $C_2\hyp\bbE_\infty$ ring, so too is $H\bbF_2^{C_2}\otimes H\bbF_2^{C_2}$, and thus there are norms
\begin{equation}\label{eq:n}
\calA_\ast^\cl =  \pi_\ast^e(H\bbF_2^{C_2}\otimes H\bbF_2^{C_2})  \rightarrow \pi_{\ast(1+\sigma)}^{C_2}(H\bbF_2^{C_2}\otimes H\bbF_2^{C_2}) =  \calA_{\ast(1+\sigma)}^{C_2}.
\end{equation}
The main point of the proof is to understand something about these.

We must first recall some of the structure of $\calA^{C_2}$; we mostly follow the treatment in \cite[Section 2]{guillouhillisaksenravenel2020cohomology}. Let $\bbM^{C_2} = \pi_\star H\bbF_2^{C_2}$ and $\bbM^{\bbR} = \pi_\star H \bbF_2^\bbR$ be the bigraded coefficient rings of $C_2$-equivariant and $\bbR$-motivic mod $2$ homology respectively. Then
\[
\bbM^\bbR = \bbF_2[\tau,\rho],\qquad \bbM^{C_2} = \bbM^{\bbR}\oplus NC,\qquad NC = \bbF_2\{\frac{\gamma}{\rho^j\tau^k}:j\geq 0,k\geq 1\},
\]
where these symbols have homological degrees
\[
|\tau| = 1-\sigma,\qquad |\rho| = -\sigma,\qquad |\gamma| = \sigma-1.
\]
Moreover,
\[
\calA^\bbR = \bbM^\bbR[\xi_1,\xi_2,\ldots,\tau_0,\tau_1,\ldots]/(\tau_i^2+\tau\xi_{i+1}+\rho(\tau_0\xi_{i+1}+\tau_{i+1})),\qquad \calA^{C_2} = \bbM^{C_2}\otimes_{\bbM^\bbR}\calA^\bbR.
\]

The right unit for $\calA^{C_2}$ restricts to define a $\calA^\bbR$-comodule structure on the summand $NC\subset\bbM^{C_2}$. Conversely, the $\calA^\bbR$-comodule structure on $NC$ makes $\bbM^{C_2}$ into a comodule algebra over $\calA^\bbR$, and this enables one to endow $\calA^{C_2} = \bbM^{C_2}\otimes_{\bbM^\bbR}\calA^\bbR$ with the structure of a Hopf algebroid. In particular, this construction extends to show that if $I\subset \bbM^{C_2}$ is an $\calA^\bbR$-comodule ideal, then the quotient $(\bbM^{C_2}/I) \otimes_{\bbM^\bbR}\calA^\bbR$ still carries the structure of a Hopf algebroid.

The norms of \cref{eq:n} are not additive, but they are additive modulo transfers. Using the $C_2$-equivariant cofiber sequence $C_{2+}\rightarrow S^0\xrightarrow{\rho}S^\sigma$, one finds that the transfer ideal equals the annihilator of the Euler class $\rho$. Explicitly, the transfers on $\bbM^{C_2}$ are given by
\[
\tr\colon \bbF_2\rightarrow \bbM^{C_2}_{n(1-\sigma)},\qquad \tr(1) = \begin{cases}0&n\geq -1,\\\frac{\gamma}{\tau^{-n-1}}&n\leq -2,\end{cases}
\]
as these are the only classes in their respective degrees killed by $\rho$. If we write $I_{\tr}$ for the transfer ideal in $\bbM^{C_2}$ or $\calA^{C_2}$, then 
\[
\calA^{C_2}/I_{\tr} \cong (\bbM^{C_2}/I_{\tr})\otimes_{\bbM^\bbR}\calA^\bbR,
\]
and this retains the structure of a Hopf algebroid.

\begin{lemma}\label{lem:tprims}
The group of primitives in $\calA^{C_2}/I_{\tr}$ is zero in degrees of the form $|\xi_n| = (2^n-1)(1+\sigma)$ for $n\geq 2$.
\end{lemma}
\begin{proof}
We may identify the primitives $\operatorname{Prim}(\calA^{C_2}/I_{\tr})$ as an $\Ext$ group:
\[
\operatorname{Prim}(\calA^{C_2}/I_{\tr}) \cong \Ext^1_{\calA^{C_2}/I_{\tr}}(\bbM^{C_2}/I_{\tr},\bbM^{C_2}/I_{\tr}) \cong \Ext^1_{\calA^\bbR}(\bbM^\bbR,\bbM^{C_2}/I_{\tr}).
\]
Abbreviate $M = \bbM^{C_2}/I_{\tr}$. The groups $\Ext^1_{\calA^\bbR}(\bbM^\bbR,M)$ may be computed via a Koszul complex of the form
\begin{center}\begin{tikzcd}
\Lambda^\bbR[0]\otimes_{\bbM^\bbR}M\ar[r,"\delta_0"]&\Lambda^\bbR[1]\otimes_{\bbM^\bbR}M\ar[r,"\delta_1"]&\Lambda^\bbR[2]\otimes_{\bbM^\bbR}M\ar[r]&\cdots
\end{tikzcd},\end{center}
where $\Lambda^\bbR$ is the $\bbR$-motivic lambda algebra; see \cite[Remark 2.3.5]{balderramaculverquigley2021motivic}. It therefore suffices to show that $\ker(\delta_1) = 0$ in degrees of the form $|\xi_n|$ for $n\geq 2$.

Note that $\Lambda^\bbR[1]\otimes_{\bbM^\bbR}M$ is generated by elements of the form $\lambda_rx$ with $r\geq 0$ and $x\in M$, and that internal algebraic degrees we have
\[
|\lambda_r| = \left\lfloor\frac{r}{2}\right\rfloor + 1 + \left\lceil \frac{r}{2}\right\rceil\sigma.
\]
Given an object $x$ with $RO(C_2)$-degree $a+b\sigma$, write $v(x) = a-b$. Then
\[
v(\xi_n) = 0,\qquad v(\tau) = 2,\qquad v(\rho) = 1,\qquad v(\gamma) = -2,\qquad v(\lambda_r) = \begin{cases}1&r\text{ even},\\0&r\text{ odd}.\end{cases}
\]
Thus if $v(\lambda_rx) = v(\xi_n) = 0$ then $x=1$ and $r$ is odd. The only such elements in $\ker(\delta_1)$ are those of the form $\lambda_{2^a-1}$ for $a\geq 1$, and these are not in the degree of $\xi_n$ for $n\geq 1$.
\end{proof}

\begin{prop}\label{prop:nt}
The norm $N\colon \calA_\ast^\cl\rightarrow\calA_{\ast(1+\sigma)}^{C_2}/I_{\tr}$ is a map of Hopf algebroids, given on generators by $N(\xi_n) = \xi_n$. In other words, $N$ is compatible with the map $P$ of \cref{eq:pa}.
\end{prop}
\begin{proof}
The structure maps in the Hopf algebroid $\calA^{C_2}$ are obtained from various $C_2\hyp\bbE_\infty$ maps between the spectra $(H\bbF_2^{C_2})^{\otimes k}$ for $k\geq 1$. It follows by naturality that norms commute with these structure maps. As $\calA^{C_2}/I_{\tr}$ is a quotient Hopf algebroid of $\calA^{C_2}$, the same is true for $N\colon \calA_\ast^\cl\rightarrow\calA_{\ast(1+\sigma)}^{C_2}/I_{\tr}$. As this map is moreover additive, it is a map of Hopf algebroids. We induct on $n$ to show that $N(\xi_n) = \xi_n$ for $n\geq 1$.

First consider $n=1$. Even before modding out by the transfer ideal, $N(\xi_1) \in \calA^{C_2}_{1+\sigma}$ must be some class lifting $\xi_1^2$ under the forgetful map $\calA_{1+\sigma}^{C_2}\rightarrow\calA_2^\cl$. The class $\xi_1\in \calA_{1+\sigma}^{C_2}$ is the only possibility.

Next let $n\geq 2$ and suppose we have verified $N(\xi_i) = \xi_i$ in $\calA^{C_2}/I_{\tr}$ for all $i < n$. As $N$ is a map of Hopf algebroids, we find
\begin{align*}
\Delta (N(\xi_n)) &=  N(\Delta(\xi_n)) = N(\sum_{0\leq i \leq n} \xi_{n-i}^{2^i} \otimes \xi_i) \\
 &= N(\xi_n)\otimes 1 + \sum_{0 < i < n} N(\xi_{n-i})^{2^i} \otimes N(\xi_i) + 1 \otimes N(\xi_n) \\
 &= N(\xi_n)\otimes 1 + \sum_{0 < i < n} \xi_{n-i}^{2^i} \otimes \xi_i + 1 \otimes N(\xi_n),
\end{align*}
where the last equality is an application of our inductive hypothesis. It follows that the difference $N(\xi_n) - \xi_n$ is primitive, and thus $N(\xi_n) = \xi_n$ by \cref{lem:tprims}.
\end{proof}

We have now all but given the following.

\begin{proof}[Proof of \cref{thm:c2}]
Consider the HLSS associated to the canonical resolution
\[
S_{C_2}\rightarrow\lim_{n\in\Delta} (H\bbF_2^{C_2})^{\otimes n+1}.
\]
This yields the $C_2$-equivariant Adams spectral sequence upon taking fixed points, and the classical Adams spectral sequence upon taking underlying spectra. This puts us squarely in the context of \cref{thm:main}, which tells us that $\Sq\colon \pi_\ast S\rightarrow \pi_{\ast(1+\sigma)}S_{C_2}$ is modeled in filtration $f$ by $\rho^f N$, with $N$ the norm for $(H\bbF_2^{C_2})^{\otimes f+1}$. When $f = 0$, the norm is simply given by $N(1) = 1$. When $f \geq 1$, as $\rho$ annihilates the transfer ideal, \cref{prop:nt} implies $\rho N = \rho P$, and the theorem follows.
\end{proof}

\section{Norms on \texorpdfstring{$\pi_0$}{pi\_0} of the equivariant \texorpdfstring{$KU$}{KU}-local sphere}\label{sec:kug}

We now consider \cref{thm:ku} and related matters.

\subsection{Preliminaries}\label{ssec:kuprelim}

We begin by recalling some background on equivariant $K$-theory. Fix for now a finite group $G$, and write $KU_G$ for the $G$-equivariant spectrum of $G$-equivariant complex $K$-theory. Equivariant Bott periodicity takes the following form: If $V\in RU(G)$ is a virtual complex $G$-representation, then there is an invertible \textit{Bott class}
\[
\beta^V \in KU_G^0(S^V).
\]
As usual there, are two natural choices of Bott classes, related by complex conjugation. With notation from \cite{atiyah1968bott}, we shall take our Bott classes to be defined by $\beta^V = \lambda_V$ when $V$ is a $G$-representation. In particular, $\beta = \beta^\bbC = 1 - \calL \in KU^0(S^2) = \pi_2 KU$, where $\calL\rightarrow S^2$ is the canonical line bundle. It is this choice that is well behaved with respect to power operations in $K$-theory (see \cref{lem:powbott}).

If $V$ is a complex $G$-representation, then the \textit{Euler class}   $e(V) \in RU(G)$ of $V$, in the sense of \cite[Chapter 7]{dieck1979transformation}, is defined as the image of $\beta^V$ under the map
\[
KU_G^0(S^V)\rightarrow KU_G(S^0) \cong RU(G)
\]
given by restriction along the inclusion of poles $S^0\rightarrow S^V$. We will discuss these further in \cref{ssec:euler}.

Now let $\Glob$ denote the homotopy theory of global equivariant spectra with respect to the family of finite groups, formalized as in \cite{schwede2018global}, and let $\textbf{KU}$ be the global spectrum of equivariant complex $K$-theory constructed in \cite[Section 6.4]{schwede2018global}. This is a refinement of the $G$-equivariant spectra $KU_G$, in the sense that there are symmetric monoidal functors $U_G\colon \Glob\rightarrow\Sp^G$ and $KU_G \simeq U_G\textbf{KU}$. For our purposes, we may take this as the \textit{definition} of the $G$-spectra $KU_G$, to be assured that the $G\hyp\bbE_\infty$ structure on $KU_G$ is compatible with the ultracommutative ring structure of $\textbf{KU}$; it is not obvious whether the $G\hyp\bbE_\infty$ structure on $KU_G$ is unique, see for instance \cite{bohmannhazelishakkedziorekmay2021genuine}.

The ultracommutative ring structure on $\textbf{KU}$ gives rise to maps $\bbP^m KU_G\rightarrow KU_{\Sigma_m\wr G}$, and this in turn induces power operations of the following form: if $X$ is a $G$-space, then $X^{\times m}$ is naturally a $\Sigma_m\wr G$-space, and there are power operations
\[
P^m\colon KU_G^0(X_+) \rightarrow KU_{\Sigma_m\wr G}^0(X^{\times m}_+).
\]
If $X$ is a based $G$-space, then one may instead consider 
\[
P^m\colon KU_G^0(X)\rightarrow KU_{\Sigma_m\wr G}^0(X^{\wedge m}).
\]
We round out this discussion by noting the following.

\begin{lemma}\label{lem:powbott}
If $V$ is a virtual complex $G$-representation, then
\[
P^m(\beta^V) = \beta^{\rho_m\otimes V} \in KU_{\Sigma_m\wr G}^0(S^{\rho_m\otimes V}).
\]
\end{lemma}
\begin{proof}
This is essentially classical, so let us just sketch how the pieces fit together. By multiplicativity, we may reduce to the case where $V$ is a complex $G$-representation. The proof of \cite[Theorems 6.3.32(iii)]{schwede2018global} extends to show that if $X$ is a $G$-space and $E\in KU_G^0(X_+)$ is the class of a vector bundle, then $P^m E \in KU_{\Sigma_m\wr G}^0 (X^{\times n}_+)$ is the class of the external tensor power $E^{\boxtimes m}$. Put another way, the power operations arising from the ultracommutative ring structure on $\textbf{KU}$ agree with the classic power operations constructed by Atiyah in \cite{atiyah1966power}. At this point, with notation from \cite{atiyah1968bott}, one computes that $P^m(\beta^V) = P^m(\lambda_V) = \lambda_V^{\otimes m} = \lambda_{\rho_m\otimes V} = \beta^{\rho_m\otimes V}$. This is where we have used our choice of Bott classes, as for example $P^2(\lambda_\bbC^\ast) = - \lambda_{\rho_2\otimes\bbC}^\ast$, still with notation from \cite{atiyah1968bott}.
\end{proof}

\subsection{The main proposition}

In the appendix (\cref{prop:ku1proof}) we verify \cref{prop:kul}. This implies, among other things, that the global ultracommutative ring spectrum $S_{\textbf{KU}}^\wedge =  \lim_{n\in\Delta} \textbf{KU}^{\otimes n+1}$ refines the $G$-spectra $L_{KU_G}S_G$, at least for $G$ a finite nilpotent group. Now define
\[
L = \pi_0 S_{\textbf{KU}}^\wedge
\]
Then $L$ is a global power functor for the family of finite groups in the sense of \cite[Chapter 5]{schwede2018global}. This means that for each finite group $G$ we are given an abelian group $L(G)$, together with restrictions along arbitrary homomorphisms, transfers along injective homomorphisms, external pairings $L(G)\otimes L(K)\rightarrow L(G\times K)$, and power operations 
\[
P^m_G\colon L(G)\rightarrow L(\Sigma_m \wr G),
\]
all subject to a number of compatibilities. When the group $G$ is clear from context, we shall write $P^m = P^m_G$. The assertion that if $G$ is nilpotent then $S_{\textbf{KU}}^\wedge$ refines the $G$-spectrum $L_{KU_G}S_G$ says, among other things, that
\[
L(G) = \pi_0^G L_{KU_G}S_G.
\]
Each $\ul{\pi}_0 L_{KU_G}S_G$ is a $G$-Tambara functor, and this is contained in the global power structure of $L$. In short, if $K\subset G$ is a subgroup of index $m$, then the norm $N_K^G\colon L(K)\rightarrow L(G)$ is recovered by postcomposing $P^m_K\colon L(K)\rightarrow L(\Sigma_m\wr K)$ with restriction along a suitable embedding $G\rightarrow\Sigma_m\wr K$. See \cite[Remark 5.17]{schwede2018global} for a more detailed discussion.

We do not know the value of $L(G)$ in general, even as a mere abelian group. When $G$ is an odd $p$-group, $L(G) \cong R\bbQ(G)[\epsilon]/(\epsilon^2,2\epsilon)$ \cite[Theorem 1.1, Proposition 6.7]{bonventreguilloustapleton2022kug}. It is also not hard to show directly that the same is true for $G = C_2$. It seems plausible that $L(G)$ might be approachable via an analysis of the $KU_G$-based Adams spectral sequence. We shall not attempt to carry out any such analysis here, but for the interested reader point out that the descent from $KU_G$ to $KO_G$ is fully described in \cite[Example 9.19]{mathewnaumannnoel2017nilpotence}, and it may be fruitful to start with $KO_G$ rather than $KU_G$.

As $L$ is equipped with restrictions along arbitrary homomorphisms, the sequence $e\rightarrow G \rightarrow e$ shows that, for every group $G$, the ring $L(G)$ is an augmented $L(e)$-algebra. We may identify $L(e)$ explicitly as
\[
L(e) = \pi_0 L_{KU}S = \bbZ[\epsilon]/(2\epsilon,\epsilon^2).
\]
In particular, the class $\epsilon$ resides in $L(G)$ for any group $G$, and we would like to understand how power operations behave on $\epsilon$. Observe that
\[
P^m_G(\epsilon) = P^m(\res^e_{G}(\epsilon)) = \res^{\Sigma_m}_{\Sigma_m\wr G}(P^m_e(\epsilon)).
\]
Thus, to determine $P^m_G(\epsilon)$, it suffices to consider the case where $G = e$, at least once the underlying global Mackey functor of $L$ is known. Although we have not computed $L(\Sigma_m)$, we can say the following.

Write $\ol{\rho}_m^\bbC$ for the reduced complex permutation representation of $\Sigma_m$.

\begin{prop}\label{prop:powepsilon}
The class $P^m(\epsilon) \in L(\Sigma_m)$ is detected in the $KU_{\Sigma_m}$-based Adams spectral sequence by $e(\ol{\rho}_m^\bbC) \cdot \epsilon$.
\end{prop}
\begin{proof}
Consider the $\textbf{KU}$-based Adams spectral sequence. This is the HLSS associated to the canonical resolution
\[
S_{\textbf{KU}}^\wedge \simeq \lim_{n\in\Delta} \textbf{KU}^{\otimes n+1},
\]
and gives, for every finite group $G$ and $\alpha\in RO(G)$, the $KU_G$-based Adams spectral sequence of signature
\[
{}^G E_1^{s+\alpha,t+\alpha} = \pi_{t+\alpha}^G(KU_G^{\otimes t-s+1}) \Rightarrow \pi_{s+\alpha}^G L_{KU_G}S_G,
\]
compatible with all restrictions and transfers. 

When $G = e$, this is the nonequivariant $KU$-based Adams spectral sequence. The class $\epsilon\in \pi_0 L_{KU}S$ is detected by some class $\epsilontilde\in{}^e E_2^{0,2}$ in filtration $2$. It follows from \cref{thm:powhfpss} that $P^m(\epsilon)$ is detected by $Q(\epsilontilde)\in {}^{\Sigma_m}E_2^{0,2}$, where $Q\colon {}^eE_2^{0,2}\rightarrow {}^{\Sigma_m}E_2^{0,2}$ is induced by
\[
a_m^2 P^m_2\colon  \pi_2^e KU \rightarrow \pi_{\rho_m^\bbC}^{\Sigma_m}KU_{\Sigma_m}\rightarrow \pi_2^{\Sigma_m}KU_{\Sigma_m}.
\]
By \cref{lem:powbott}, we may identify this as
\[
\bbZ\{\beta\}\rightarrow RU(\Sigma_m)\{\beta^{\rho_m^\bbC}\}\rightarrow RU(\Sigma_m)\{\beta\},
\]
where the first map acts by $\beta\mapsto \beta^{\rho_m^\bbC}$ and the second map acts by $\beta^{\rho_m^\bbC}\mapsto e(\ol{\rho}_m^\bbC)\cdot\beta$. More succinctly, the map $Q$ is given by multiplication with $e(\ol{\rho}_m^\bbC)$, and the proposition follows.
\end{proof}

\subsection{Euler classes}\label{ssec:euler}

To translate from \cref{prop:powepsilon} to \cref{thm:ku}, we must recall some information about Euler classes. Let $G$ be a finite group. Given a complex $G$-representation $V$, the Euler class $e(V)$ may be identified explicitly as
\[
e(V) = \sum_n (-1)^n \Lambda^n(V) \in RU(G).
\]
This follows from the definition of $e(V)$ and the construction of the Bott class $\beta^V$, see for instance \cite[IV \S 1]{atiyahtall1969group}.
In particular, write $\Cl(G;\bbC)$ for the ring of class functions on $G$, and for $V\in RU(G)$ write $\chi(V,\bs)\in \Cl(G;\bbC)$ for its character. For a complex $G$-representation $V$ and $g\in G$, write $f(V,g)(t)\in \bbC[t]$ for the characteristic polynomial of the linear map $g\colon V\rightarrow V$. Then we obtain the following identification of the character of an Euler class.

\begin{lemma}\label{lem:charactereuler}
Let $V$ be a complex $G$-representation. Then $\chi(e(V),g) = f(V,g)(1)$.
\end{lemma}
\begin{proof}
The claim is that the characteristic polynomial of $g\colon V\rightarrow V$ evaluated at $1$ agrees with the alternating sum of the traces of $g\colon \Lambda^n V\rightarrow\Lambda^n V$. This is a standard fact from linear algebra, see for instance \cite[\S 8, no.\ 11]{bourbaki1948algebrai}.
\end{proof}

Given a finite $G$-set $X$, let
\[
\widetilde{\bbC}[X] = \coker(\bbC\rightarrow \bbC[X])
\]
be the associated reduced permutation representation, and set
\[
e(X) = e(\widetilde{\bbC}[X]) \in RU(G).
\]
We then have the following.

\begin{prop}\label{prop:charactereuler}
${}$
\begin{enumerate}
\item Given a finite $G$-set $X$, we have $e(X)\neq 0$ if and only if there exists some $g\in G$ such that the cyclic group $\langle g \rangle$ acts transitively on $X$.
\item Let $p$ be a prime, and suppose that $K\subset G$ is a normal subgroup such that $G/K$ is cyclic of order $p^n$. Let $N \subset G$ be the unique subgroup of index $p$ containing $K$. Then $e(G/K) = p^{n-1} (p\bbC- \bbC[G/N])$.
\item In particular, in the situation of (2), if $p$ is odd then $e(G/K)\equiv \widetilde{\bbC}[G/N]\pmod{2}$.
\end{enumerate}
\end{prop}
\begin{proof}
(1)~~Given $g\in G$, we may identify
\[
f(\widetilde{\bbC}[X],g)(t) = \frac{f(\bbC[X],g)(t)}{1-t}.
\]
It follows from \cref{lem:charactereuler} that $\chi(\widetilde{\bbC}[X],g) \neq 0$ if and only if $1$ is not a repeated root of $f(\bbC[X],g)(t)$. The element $g$ acts on $\bbC[X]$ by a permutation matrix, and an elementary computation shows that this holds if and only if $g$ acts transitively on $X$. The claim follows as $e(X) \neq 0$ if and only if $\chi(\widetilde{\bbC}[X],g) \neq 0$ for some $g\in G$.

(2)~~Write $q\colon G\rightarrow G/K\cong C_{p^n}$. Then $e(G/K) = q^\ast e(C_{p^n})$, so we may reduce to the case where $K = e$ and $G = C_{p^n}$. An elementary computation, following the ideas in (1), shows that
\[
\chi(e(C_{p^n}),g) = \begin{cases}p^n&\text{if }g\text{ generates }C_{p^n};\\
0&\text{otherwise}.
\end{cases}
\]
A second elementary computation shows that $p^{n-1}(p\bbC - \bbC[C_{p^n}/C_{p^{n-1}}])$ has the same character, implying that $e(C_{p^n}) = p^{n-1}(p \bbC - \bbC[C_{p^n}/C_{p^{n-1}}])$ as claimed.

(3)~~If $p$ is odd, then $p^{n-1}(p\bbC - \bbC[G/N])\equiv \bbC - \bbC[G/N]\equiv \widetilde{\bbC}[G/N]\pmod{2}$ in $RU(G)$.
\end{proof}

\subsection{The proof of \cref{thm:ku}}\label{ssec:ku}

Note the following immediate corollary of \cref{prop:powepsilon} and the interaction between the power operations $P^m$ and the norms $N_K^G$.

\begin{cor}[of \cref{prop:powepsilon}]\label{cor:kugk}
Let $K\subset G$ be a subgroup. Then $N_K^G(\epsilon) \in L(G)$ is detected in the $KU_G$-based Adams spectral sequence  by $e(G/K)\cdot\epsilon$.
\qed
\end{cor}

We are now in a position to prove \cref{thm:ku}. Let us again recall the main players. Fix an odd prime $p$. In \cite[Theorem 1.1, Proposition 6.7]{bonventreguilloustapleton2022kug}, Bonventre--Guillou--Stapleton prove that if $G$ is a $p$-group, then there is an isomorphism
\[
\ul{\pi}_0 L_{KU_G}S_G \cong \ul{R}\bbQ[\epsilon]/(2\epsilon,\epsilon^2)
\]
of Green functors, where $\ul{R}\bbQ$ is the Green functor whose value at a subgroup $K\subset G$ is the rational representation ring $R\bbQ(K)$. In our context, this says that if $G$ is any $p$-group, then
\begin{equation}\label{eq:jjj}
L(G) \cong R\bbQ(G)[\epsilon]/(2\epsilon,\epsilon^2).
\end{equation}
This easily extends to an identification of the restriction of the global Green functor $L$ to the family of $p$-groups. We now give the following.

\begin{proof}[Proof of \cref{thm:ku}]
Consider the norm
\[
N_K^G\colon R\bbQ(K)[\epsilon]/(2\epsilon,\epsilon^2)\rightarrow R\bbQ(G)[\epsilon]/(2\epsilon,\epsilon^2).
\]
By \cref{cor:kugk}, $N_K^G(\epsilon)$ is detected in the $KU_G$-based Adams spectral sequence by $e(G/K)\cdot \epsilon$. As all elements of $R\bbQ(G)[\epsilon]/(2\epsilon,\epsilon^2)$ are detected either on the $0$-line, as an element of $R\bbQ(G)$, or on the $2$-line, as the product of an element of $R\bbQ(G)$ with $\epsilon$, we may deduce that $N_K^G(\epsilon) = e(G/K)\cdot\epsilon$ on the nose. The final claims regarding the case where $K\subset G$ is normal follow from \cref{prop:charactereuler}.
\end{proof}

\section{Power operations for the \texorpdfstring{$K(1)$}{K(1)}-local sphere}\label{sec:k1}

This section carries out the computation promised in \cref{thm:sk1}.

\subsection{Generalities on power operations}\label{ssec:powprelim}

We begin by recalling some basic properties of power operations, cf.\ \cite[Chapter VIII]{brunermaymccluresteinberger1986hinfinity}. Fix a prime $p$, and for a spectrum $R$ define
\[
\pi_{s,w}b(R) = [\Sigma^{s-pw}(S^w)^{\otimes p}_{\h \Sigma_p},R] = R^{(p-1)w-s}\Th(w \ol{\rho}_p\downarrow B\Sigma_p) = \pi_{s-pw+w\rho_p}F(E\Sigma_p,i_\ast R).
\]
These are all different names for the same object; the third term is the $R$-cohomology of the Thom spectrum of a multiple of the reduced permutation representation $\ol{\rho}_p$ of $\Sigma_p$, and the fourth term is a piece of the $\Sigma_p$-equivariant spectrum obtained as the Borel construction on $R$ with trivial action. There are maps
\begin{gather*}
a\colon \pi_{s,w}b(R)\rightarrow\pi_{s-(p-1)w,w-1}b(R), \qquad i\colon \pi_s R\rightarrow \pi_{s,0}b(R),\\
 \res_w\colon \pi_{s,w}b(R)\rightarrow\pi_s R,\qquad \tr_w\colon \pi_s R\rightarrow\pi_{s,w}b(R),
\end{gather*}
given by multiplication with the Euler class of $\ol{\rho}_p$, inflation, restriction, and transfer. We shall write $\tr = \tr_w$ and $\res = \res_w$ when $w$ is clear from context, and shall use $i$ to regard $\pi_\ast R$ as a subobject of $\pi_{\ast,0}b(R)$.

Now suppose that $R$ is a $p$-local $\bbE_\infty$ ring. Then the pair $(\pi_\ast R,\pi_{\ast,\ast}b(R))$ is a good device for understanding power operations on $R$. The $p$th total power operation for $R$ takes the form
\[
P\colon \pi_n R\rightarrow \pi_{pn,n}b(R),
\]
and the behavior of $P$ may be encoded in structure present on $\pi_{\ast,\ast}b(R)$, as we now recall.

First, $\pi_{\ast,\ast}b(R)$ is a bigraded ring, and $P$ is multiplicative, i.e.\ $P(xy) = P(x)P(y)$ for $x\in \pi_n R$ and $y\in \pi_m R$. Second, define
\[
h[w] = \frac{\tr_w(1)}{(p-1)!} \in \pi_{0,w}R,
\]
and abbreviate $h = h[0]$. These elements satisfy
\[
a\cdot h[w] = 0.
\]
Let $C(x,y) = p^{-1}((x+y)^p-x^p-y^p)$. Then for $x,y \in \pi_n R$, we have
\[
P(x+y) = P(x) + C(x,y)\cdot h[n] + P(y).
\]
In particular, for $k\in\bbZ$ we have
\begin{equation}\label{eq:pk}
P(k) = k - \frac{k-k^p}{p}h.
\end{equation}
Third, we note that $a P$ is additive, and if $R$ arises as a limit of $\bbE_\infty$ rings then $P$ is modeled in filtration $f$ of the associated HLSS by $a^f P$, as described in \cref{thm:powhfpss}.

\subsection{Morava \texorpdfstring{$E$}{E}-theory}
Let $E$ be a Morava $E$-theory with formal group $\bbG\rightarrow \Spf E_0$. See \cite{peterson2018formal} for a textbook reference. We wish to describe the general shape of $\pi_{\ast,\ast}b(E)$. Let
\[
\omega = \pi_2 E,\qquad R = \pi_{0,0}b(E) = E^0 B\Sigma_p,\qquad L = \pi_{2(p-1),2}b(E) = E^0 \Th(\bbC\otimes \ol{\rho}_p).
\]
Then $R$ is a commutative $E_0$-algebra and $L$ is an invertible $R$-module. Writing $L^n = L^{\otimes_R n}$, we have the following picture:
\begin{center}\begin{tikzcd}
R\ar[r,"a^2"]\ar[d]\ar[dr,"a"]&L^{-1}\\
R/(h)\ar[r,dashed,"\cong"]&\pi_{-(p-1),-1}E\ar[u,"a",tail]
\end{tikzcd}.\end{center}
That the Euler class $a$ annihilates $h$ is standard, and that the resulting map $R/(h)\rightarrow \pi_{-(p-1),-1} b(E)$ is an isomorphism may be found in dual form in \cite[Proposition 7.2, Remark 7.4]{rezk2009congruence}. Note in particular that postcomposing the product on $R/(h)$ with $a$ gives a map
\begin{equation}\label{eq:modh}
R/(h)\otimes R/(h)\rightarrow R/(h) \rightarrow L^{-1}.
\end{equation}

\begin{lemma}\label{lem:lubintatepowerring}
There are isomorphisms
\begin{align*}
\pi_{2a(p-1)+2b,2a} b(E) &= \omega^b \otimes_{E_0} L^{ a},\\
\pi_{(2a+1)(p-1)+2b,2a+1}b(E) &= \omega^b \otimes_{E_0} L^{a} \otimes_R R/(h),
\end{align*}
all other degrees being zero. The ring structure is induced by the canonical isomorphisms $(\omega^b\otimes_{E_0} L^{a})\otimes_R (\omega^{b'}\otimes_{E_0} L^{a'}) \cong \omega^{b+b'}\otimes_{E_0} L^{ a+a'}$, applying the Euler class as in \cref{eq:modh} when needed.
\end{lemma}
\begin{proof}
The Morava $E$-theory of $B\Sigma_p$ is concentrated in even degrees \cite[Theorem E]{hopkinskuhnravenel2000generalized}. The lemma then combines the above discussion with the Thom isomorphisms $\pi_{(a+2a')(p-1)+2(b+b'),(a+2a')}b(E)\cong \pi_{a(p-1)+2b,a}b(E)\otimes_R \pi_{2a'(p-1)+2b,2a'}b(E)$.
\end{proof}

\begin{prop}\label{prop:adams}
The Adams operation $\psi^k$ for $k\in\bbZ_p^\times$ acts on $L^{ a}$ by multiplication with $(1-\frac{1}{p}(1-k^{a(p-1)})h)$.
\end{prop}
\begin{proof}
First we consider the case $a = 0$, where $L^{0} = R = E^0 B\Sigma_p$. Here we are claiming that $\psi^k$ acts trivially on $E^0 B\Sigma_p$. By Strickland's theorem \cite{strickland1998morava}, $R/(h)$ is the $E_0$-algebra classifying rank $p$ subgroups of $\bbG$. The Adams operation $\psi^k$ corresponds to the automorphism $[k]\colon \bbG\rightarrow\bbG$ defined over $\Spf E_0$. This fixes all subgroups of $\bbG$, and so $\psi^k$ acts trivially on $R/(h)$. As the transfer is split $K(n)$-locally \cite{clausenmathew2017short}, it follows that $\psi^k$ acts trivially on $R$.

Now consider general $a$. As $\psi^k$ acts on $\omega^a$ by multiplication with $k^a$, it suffices to show that $\psi^k$ acts on $\omega^a \otimes L^{a}$ by multiplication with $k^a(1-\tfrac{1}{p}(1-k^{a(p-1)})h)$. Observe that this is exactly the element $P(k^a)$ seen in \cref{eq:pk}. Choose a generator $u\in \omega$, so that we have $P(u^a) \in \omega^a \otimes_{E_0} L^a$. As $P(u^a)\cdot P(u^{-a}) = P(1) = 1$, we find that $P(u^a)$ gives a trivialization of the invertible $R$-module $\omega^a\otimes_{E_0}L^a$. As $\psi^k$ acts trivially on $R$, it thus suffices to show that $\psi^k(P(u^a)) = P(k^a) P(u^a)$. Indeed, as $\psi^k$ acts on $E$ by $\bbE_\infty$ automorphisms, we have
\[
\psi^k(P(u^a)) = P(\psi^k(u^a)) = P(k^a u^a) = P(k^a)P(u^a)
\]
as needed.
\end{proof}

Now write $KU_p$ for the spectrum of $p$-adic complex $K$-theory. We consider complex $K$-theory to be oriented as described in \cref{ssec:kuprelim}.

\begin{prop}\label{thm:kpow}
Let $\tau^{-2} \in \pi_{0,2}b(KU_p)$ be the Thom class of $\ol{\rho}_p^\bbC$, i.e.\ the Bott class of $\rho_p^\bbC - p\bbC$, and abbreviate $d = a^2 \beta \tau^{-2} \in \pi_{0,0}b(KU_p)$. Then $h=p-d$ and
\[
\pi_{\ast,\ast}b(KU_p) \cong \bbZ_p[\beta^{\pm 1},\tau^{\pm 2},a]/(ah).
\]
The Adams operation $\psi^k$ for $k\in\bbZ_p^\times$ acts by ring automorphisms, and is determined by
\[
\psi^k(\beta) = k \beta,\qquad \psi^k(\tau^2) = \tau^2(1+\tfrac{1}{p}(k^{p-1}-1) d)
\]
Power operations are determined by general properties and
\[
P(\beta) = \beta^{p}\tau^{-2}.
\]
\end{prop}
\begin{proof}
Restriction along $C_p\subset \Sigma_p$ identifies $KU_p^0 B\Sigma_p$ as the subring of $KU_p^0 BC_p$ fixed under the action of $\Aut(C_p)$. It follows quickly that $h$ is the image of the permutation representation $\rho_p^\bbC$ under the completion map $R(\Sigma_p)\rightarrow KU_p^0 B\Sigma_p$, and that
\[
KU_p^0 B\Sigma_p \cong \bbZ_p[h]/(h^2-ph).
\]
On the other hand, $d = a^2 \beta \tau^{-2}$ is the Euler class of $\ol{\rho}_p^\bbC$. By \cref{prop:charactereuler}, both $d$ and $p-h$ have the same image in $KU_p^0 BC_p$, and thus $d = p-h$ in $KU_p^0 B\Sigma_p$. It follows that
\[
KU_p^0 B\Sigma_p\cong \bbZ_p[d]/(dh).
\]
The full identification of $\pi_{\ast,\ast}KU_p$ then follows from the recipe of \cref{lem:lubintatepowerring}, where now we have fixed trivializations  $\beta^{a(p-1)+b}\tau^{-2a}\in \omega^b\otimes_{E_0} L^{a}$. The action of the Adams operations was given in \cref{prop:adams}. The identity $P(\beta) = \beta^p \tau^{-2}$ was given in \cref{lem:powbott}. In this Borel context, it may also be regarded as a consequence of the fact that the map $MUP \rightarrow KU$ classifying our choice of periodic complex orientation is $\bbH_\infty$.
\end{proof}

\subsection{Odd primes}\label{ssec:odd}

Let $p$ be an odd prime and fix a topological generator $k\in\bbZ_p^\times$. Then there is a fiber sequence
\begin{equation}\label{eq:j}
\begin{tikzcd}
S_{K(1)}\ar[r]&KU_p\ar[r,"\psi^k-1"]&KU_p
\end{tikzcd}.\end{equation}
We can use this to easily compute power operations for $S_{K(1)}$. For a spectrum $X$ and class $x\in KU_p^n X$, write $[x] \in S_{K(1)}^{n+1}X$ for the image of $x$ under the boundary map associated to \cref{eq:j}. We then have
\[
\pi_0 S_{K(1)} = \bbZ_p,\qquad \pi_{2n-1}S_{K(1)} = \bbZ_p/(k^n-1)\{[\beta^n]\},
\]
all other groups being zero. So it suffices to compute $P([\beta^n]) \in \pi_{(2n-1)p,2n-1} b(S_{K(1)})$.

\begin{lemma}
We have
\[
\pi_{(2n-1)p,2n-1} b(S_{K(1)}) = \bbZ_p/(k^n-1)\{[a \beta^{pn}\tau^{-2n}]\}.
\]
\end{lemma}
\begin{proof}
By \cref{thm:kpow}, we have
\[
\pi_{\ast,2n-1}b(KU_p) \cong \bbZ_p[\beta^{\pm 1}]\{a\beta^{pn}\tau^{-2n}\},\qquad \psi^k(a \beta^{pn}\tau^{-2n}) = k^n a \beta^{pn}\tau^{-2n}.
\]
The lemma then follows from \cref{eq:j}.
\end{proof}

\begin{theorem}\label{thm:ppow}
The $p$th total power operation
\[
P\colon \pi_{2n-1}S_{K(1)}\rightarrow \pi_{(2n-1)p,2n-1}b(S_{K(1)})
\]
is additive, and satisfies
\[
P([\beta^n]) = [a\beta^{pn}u^{-n}].
\]
\end{theorem}
\begin{proof}
The long exact sequence associated to the fibering of \cref{eq:j} can be interpreted as the HFPSS 
\[
H^\ast(\bbZ\{\psi^k\};KU_p) \Rightarrow \pi_\ast KU_p^{\h \bbZ\{\psi^k\}} \cong \pi_\ast S_{K(1)}.
\]
As $P(\beta^n) = \beta^{pn} \tau^{-2n}$, it follows from \cref{thm:powhfpss} that $P(k[\beta^n])$ is detected in the HFPSS by $k[a \beta^{pn}\tau^{-2n}]$ for $k\in\bbZ$. As there is nothing in higher filtration, we must have $P(k[\beta^n]) = k[a\beta^{pn}\tau^{-2n}]$ on the nose.
\end{proof}

\subsection{Even primes}

Now consider $p=2$. There is a fiber sequence
\begin{equation}\label{eq:jo}
\begin{tikzcd}
S_{K(1)}\ar[r]&KO_2\ar[r,"\psi^3-1"]&KO_2
\end{tikzcd},\end{equation}
and we may compute power operations for $S_{K(1)}$ following the same approach as for odd primes, only by descent from $KO_2$ rather than $KU_p$. We begin by recalling the structure of the former. Write $\eta_{C_2}\in \pi_{1,1}b(S)$ for the $C_2$-equivariant Hopf map. This is characterized by 
\[
h = 2 + a \eta_{C_2}.
\]
Also write $\eta_{\cl}\in \pi_1 S\subset \pi_{1,0}b(S)$ for the nonequivariant Hopf map.

\begin{lemma}\label{lem:ko}
Write
\[
\pi_\ast KO_2 = \bbZ_2[\beta^{\pm 4},2\beta^2,\eta_{\cl}]/(2\cdot\eta_{\cl},2\beta^4\cdot\eta_{\cl},\eta_{\cl}^3,(2\beta^2)^2-4\beta^2).
\]
Then
\[
\pi_{\ast,\ast}b(KO_2) = \bbZ_2[\beta^{\pm 4},\tau^{\pm 4},a,\eta_{C_2},\tau^2h,2\beta^2,\beta^2\tau^2h,\eta_{\cl}]/I,
\]
where $I$ is generated by a number of relations, including $\rho\cdot\tau^{-2}\beta^2 h = \eta_{C_2}\eta_{\cl}^2$.
The Adams operation $\psi^3$ fixes all torsion classes, and otherwise is determined by the map $\pi_{\ast,\ast}b(KO_2)\rightarrow \pi_{\ast,\ast}b(KU_2)$, which sends to classes to classes of the same name, only where moreover $\eta_{\cl}\mapsto 0$ and $\eta_{C_2}\mapsto -a\beta\tau^{-2}$. The norms $P\colon \pi_n KO_2\rightarrow \pi_{2n,n}b(KO_2)$ are determined by general properties and
\[
P(\beta^4) = \beta^8\tau^{-8},\qquad P(2\beta^2) = (4+a\eta_{C_2})\beta^4\tau^{-4},\qquad P(\eta_{\cl}) = \eta_{\cl}\eta_{C_2}.
\]
\end{lemma}
\begin{proof}
See for instance \cite{balderrama2021borel} or \cite{balderrama2022ktheory}; for the former note that our $\pi_{s,w}b(KO_2)$ is its $\pi_{s,s-w}b(KO)$ and $\eta_{C_2} = -\eta_0$, and for the latter note $\pi_{s,w}b(KO_2) = \pi_{(s-w)+w\sigma} KO_{C_2}\otimes\bbZ_2$ and $\eta_{C_2} = -\eta_\sigma$.
\end{proof}

One may compute the groups $\pi_{\ast,\ast}b(S_{K(1)})$ from this, using \cref{eq:jo}. This computation was also carried out in \cite{balderrama2021borel}, but our situation is much simpler: the hard work there was to pin down the ring structure on $\pi_{\ast,\ast}b(S_{K(1)})$, which we don't need, and even for the additive structure we need only the particular groups $\pi_{2n,n}S_{K(1)}$. Because we need details of the computation, it is easier to just proceed directly.

Write $[x]$ for classes in $S_{K(1)}$-cohomology detected in the boundary of \cref{eq:jo}. Define
\begin{gather*}
\rho_n = [\beta^{4n}] \in \pi_{8n-1}S_{K(1)},\qquad \xi_n = [2\beta^{4n+2}]\in \pi_{8n+3}S_{K(1)},\qquad \mu_n = \beta^{4n}\eta_{\cl}\in \pi_{8n+1}S_{K(1)},\\
\rho_{n,n} = [\tau^{-4n}\beta^{4n}]\in \pi_{8n-1,4n}b(S_{K(1)}),\qquad \mu_{n,n} = \tau^{-4n}\mu_n \in \pi_{8n+1,4n} b(S_{K(1)}),
\end{gather*}
These names are chosen to be compatible with \cite{balderrama2021borel}; note, however, that $\pi_{s,w}$ here is $\pi_{s,s-w}$ there, and that we write $a$ instead of $\omega_0$ below. A choice must be made here: in writing $\mu_{n} = \beta^{4n}\eta_{\cl}$, we mean that $\mu_{n}$ is some class detected by $\beta^{4n}\eta_{\cl}\in \pi_{8n+1}KO_2$, and there are two such classes, and likewise for $\mu_{n,n}$. This choice is relevant to the indeterminacy in \cref{thm:2pow} and carefully handled in \cite{balderrama2021borel}, but for our purposes it does not matter what choice is made.

\begin{lemma}
The nonzero homotopy groups of $S_{K(1)}$ are $\pi_0 S_{K(1)} = \bbZ_2\{1\}\oplus\bbZ/(2)\{\eta_{\cl} \rho_0\}$, and otherwise
\[
\pi_i S_{K(1)} = \begin{cases}
\bbZ_2/(3^{4a}-1)\{\rho_n\}&i=8n-7;\\
\bbZ/(2)\{\eta_{\cl} \rho_n\}&i = 8n; \\
\bbZ/(2)\{\eta_{\cl}^2 \rho_n,\mu_n\}&i=8n+1;\\
\bbZ/(2)\{\eta_{\cl} \mu_n\}&i=8n+2;\\
\bbZ/(8)\{\xi_n\}&i=8n+3.\\
\end{cases}
\]
Moreover, we have
\[
\pi_{2i,i}b(S_{K(1)}) = \begin{cases}
\bbZ_2/(3^{2n}-1)\{a\rho_{n,n}\}&i=4n-1;\\
\bbZ/(2)\{\eta_{\cl} \rho_{n,n}, a \eta_{\cl} \eta_{C_2} \rho_{n,n} \} & i = 4n;\\
\bbZ/(2)\{\eta_{\cl}^2\eta_{C_2} \rho_{n,n},\mu_{n,n}\eta_{C_2}\}&i=4n+1;\\
0&i=4n+2;
\end{cases}
\]
only with an additional summand of the form $\bbZ_2\{1,a\eta_{C_2}\}$ in $\pi_{0,0}b(S_{K(1)})$.
\end{lemma}
\begin{proof}
These follow by a direct computation from \cref{eq:jo}, only one must verify that $\pi_{8n+1}S_{K(1)}\neq \bbZ/(4)$ and $\pi_{8n+2,4n+1}S_{K(1)} \neq \bbZ/(4)$, for which we cite \cite[Theorem 8.15]{ravenel1984localization} and \cite[Lemma 3.3.3]{balderrama2021borel}.
\end{proof}

\begin{theorem}\label{thm:2pow}
The symmetric squares
\[
P\colon \pi_n S_{K(1)}\rightarrow\pi_{2n,n}S_{K(1)}
\]
are additive for $n\neq 0$, and satisfy
\begin{align*}
P(\rho_n) &= a \rho_{2n,2n} \\
P(\eta_\cl \rho_n) &= a \eta_\cl \eta_{C_2} \rho_{2n,2n} \\
P(\eta_\cl^2 \rho_n) &= 0 \\
P(\mu_n) &\in \mu_{2n,2n}\eta_{C_2} + \bbZ/(2)\{\eta_\cl^2 \eta_{C_2} \rho_{2n,2n}\} \\ 
P(\eta_\cl \mu_n) &= 0 \\
P(\xi_n) &= 2 a \rho_{2n+1,2n+1}.
\end{align*}
Moreover, 
\[
P(\eta_\cl) = \eta_{\cl}\eta_{C_2}(1+\eta_\cl \rho_0).
\]
\end{theorem}
\begin{proof}
Both $P(\rho_n)$ and $P(\xi_n)$ may be computed along the same lines as the odd-primary case. Combining \cref{thm:powhfpss} and \cref{lem:ko}, we find
\[
P(\rho_n) = P([\beta^{4n}]) = [a\beta^{8n}\tau^{-8n}] = a \rho_{2n,2n}
\]
and
\[
P(\xi_n) = P([2\beta^{4n+2}]) = [a(2+h)\beta^{8n+4}\tau^{-8n-4}] = 2a [\beta^{8n+4}\tau^{-8n-4}] = 2a \rho_{2n+1,2n+1}.
\]
That $P(\mu_n) \in \mu_{2n,2n}\eta_{C_2} + \bbZ/(2)\{\eta_\cl^2\eta_{C_2}\rho_{2n,2n}\}$ follows by comparison with $KO_2$. Despite the indeterminacy, this is sufficient to deduce the remaining the remaining values of $P$ by multiplicativity. We have been unable to resolve this indeterminacy in general, but can describe what happens in the case $n = 0$.

There is a Hurewicz map $\pi_{\ast,\ast}S_{C_2}\rightarrow \pi_{\ast,\ast}S_{K(1)}$ from the $C_2$-equivariant stable stems, compatible with all power operations, sending a $C_2$-equivariant map $f\colon S^{a+b\sigma}\rightarrow S^0$ to the induced map $(S^{a+b\sigma})_{\h C_2} \simeq \Sigma^{a-b} (S^b)^{\otimes 2}_{\h \Sigma_2} \rightarrow S \rightarrow S_{K(1)}$. In $\pi_{\ast,\ast}S_{C_2}$, we have
\[
P(\eta_{\cl}) = \eta_{\cl}\eta_{C_2} + a \nu_{C_2},
\]
where $\nu_{C_2}\in\pi_{3,2}S_{C_2}$ is the $C_2$-equivariant quaternionic Hopf fibration (\cref{ex:sqeta}; note $a = \rho$). As $S_{K(1)}$ detects the nonequivariant quaternionic Hopf fibration, $b(S_{K(1)})$ must detect $\nu_{C_2}$. We may compute from \cref{eq:jo} and \cref{lem:ko} that
\[
\pi_{3,2}b(S_{K(1)}) = \bbZ/(8)\{[\tau^{-2}\beta^2h]\},
\]
and so the only possibility is that $\nu_{C_2}$ is detected by some odd multiple of $[\tau^{-2}\beta^2 h]$. As
\[
a\cdot [\tau^{-2}\beta^2h] = [a\cdot\tau^{-2}\beta^2h] = [\eta_{\cl}^2\eta_{C_2}] = \eta_\cl^2 \eta_{C_2} \rho_{0},
\]
the identity $P(\eta_{\cl}) = \eta_{\cl}\eta_{C_2}(1+\eta_\cl \rho_{0})$ follows.
\end{proof}

\begin{rmk}\label{rmk:theta}
If $R$ is any $K(1)$-local $\bbE_\infty$ ring, then there is a natural isomorphism
\[
\pi_{0,0}b(R) = \pi_0 R\{1,h\}
\]
Following \cite{hopkins2014k1}, if we define $\theta\colon \pi_0 R\rightarrow\pi_0 R$ by declaring $-\theta(x)$ to be the coefficient of $h$ in $P(x)$, then $\theta$ makes $\pi_0 R$ into a $\theta$-ring. This applies at any prime $p$, but let us continue focusing on $p=2$. Write $\epsilon = \eta_\cl \rho_0$, so that $\pi_0 S_{K(1)} = \bbZ_2[\epsilon]/(2\epsilon,\epsilon^2)$. As $a\eta_{C_2}\equiv -h \pmod{2}$, it follows from \cref{thm:2pow} that the action of $\theta$ on $\pi_0 S_{K(1)}$ satisfies
\[
\theta(\epsilon) = \epsilon.
\]
In fact this already follows from \cref{prop:powepsilon}. This yields an alternate proof of \cite[Theorem 5.4.8]{carmeliyuan2021higher}, using completely different methods.
\tqed
\end{rmk}

\begin{appendix}

\section{Equivariant Bousfield localizations}\label{app:bous}

This appendix, which may be read independently of the rest of the paper, gives some general material on Bousfield localizatons in equivariant stable homotopy theory. See especially \cite{hill2019equivariant,carrick2020smashing} for prior work on the topic; our approach differs in that we focus primarily on the role of nilpotent completion. Insofar as the body of the paper is concerned, this appendix contains the proof of \cref{prop:kul} (in \cref{prop:ku1proof}).

\subsection{Bousfield localizations}

We begin by reviewing some of the general theory of Bousfield localizations. Nothing in this subsection is new, we just collect everything we need in one place and in the form most convenient for us. In particular, most of this material is either routine or may be found in some form in \cite{bousfield1979localization, hoveypalmieristrickland1997axiomatic, mathew2015thick, mathew2018examples}. Fix for this subsection a presentable symmetric monoidal stable $\infty$-category $\calM$ with unit denoted $S$, together with an object $R\in\calM$.

\begin{defn}
Fix $X\in\calM$.
\begin{enumerate}
\item $X$ is \textit{$R$-acyclic} if $R\otimes X\simeq 0$;
\item $X$ is \textit{$R$-local} if $\calM(C,X)\simeq 0$ for any $R$-acyclic $C$;
\item $X\rightarrow Y$ is an \textit{$R$-equivalence} of $R\otimes X\rightarrow R\otimes Y$ is an equivalence;
\item The \textit{$R$-localization of $X$} is an $R$-local object $L_R X\in\calM$ equipped with an $R$-equivalence $X\rightarrow L_R X$.
\item The \textit{Bousfield class} of $R$ shall be the class $\langle R \rangle = \{X\in\calM:R\otimes X\}$ of $R$-acyclics.
\tqed
\end{enumerate}
\end{defn}

Observe that $R$-localization depends only on the Bousfield class of $R$, and that $\langle R \rangle \subset \langle T \rangle$ when there is a natural transformation $L_R\rightarrow L_T$. The functor of $R$-localization is lax symmetric monoidal, and in particular there is a natural map $X\otimes L_R S\rightarrow L_R X$ for each $X\in\calM$.

\begin{defn}
$R$-localization is \textit{smashing} if $X \otimes L_R S\simeq L_R X$ for all $X\in\calM$.
\tqed
\end{defn}

Suppose from now on that $R$ carries a unital product; we shall just say that $R$ is a ring. Let $\ol{R} = \Fib(S\rightarrow R)$, let $A(R) = \{\ol{R}{}^{\otimes s}\}$ be the $R$-Adams tower \cite[\S 5]{bousfield1979localization}, and let $C(R) = \{\Cof(\ol{R}{}^{\otimes s}\rightarrow S)\}$ be the associated tower under $S$.

\begin{defn}
\hphantom{blank}
\begin{enumerate}
\item The \textit{$R$-nilpotent completion} of $X\in\calM$ is $X_R^\wedge = \lim(X\otimes C(R))$.
\item Say that $X$ is \textit{$R$-convergent} if the natural map $L_R X\rightarrow X_R^\wedge$ is an equivalence, or equivalently if the natural map $X\rightarrow X_R^\wedge$ is an $R$-equivalence.
\tqed
\end{enumerate}
\end{defn}

\begin{lemma}\label{lem:nilreflect}
Let $\calN$ be another presentably symmetric monoidal stable $\infty$-category, and let $F\colon\calM\rightarrow\calN$ be a symmetric monoidal, conservative, and limit-preserving functor. Then
\begin{enumerate}
\item $F(X_R^\wedge)\simeq F(X)_{F(R)}^\wedge$ for $X\in\calM$.
\item If $F(X)$ is $F(R)$-convergent, then $X$ is $R$-convergent, and $F(L_RX)\simeq L_{F(R)}F(X)$.
\item If $F(R)$-localization is smashing and $F(X)$ is $F(R)$-convergent for all $X\in\calM$, then $R$-localization is smashing.
\end{enumerate}
\end{lemma}
\begin{proof}
(1)~~As $F$ is limit-preserving, it is exact. As $F$ is symmetric monoidal and exact, $F(X\otimes C(R))\simeq F(X)\otimes C(F(R))$. Thus $F(X_R^\wedge) = F(\lim X\otimes C(R)) \simeq \lim \left(F(X)\otimes C(F(R))\right) = F(X)_{F(R)}^\wedge$.

(2)~~Suppose that $F(X)$ is $F(R)$-convergent. We must show that $R\otimes X \rightarrow R\otimes X_R^\wedge$ is an equivalence. As $F$ is conservative, it suffices to show that $F(R\otimes X)\rightarrow F(R\otimes X_R^\wedge)$ is an equivalence. As $F$ is symmetric monoidal, and by (1), this map is $F(R)\otimes F(X)\rightarrow F(R)\otimes F(X)_{F(R)}^\wedge$, which is an equivalence as $F(X)$ is $F(R)$-convergent. Thus $X$ is $R$-convergent, and $F(L_R X)\simeq F(X_R^\wedge)\simeq F(X)_{F(R)}^\wedge\simeq L_{F(R)}F(X)$.

(3)~~Suppose that $F(R)$-localization is smashing and that $F(X)$ is $F(R)$-convergent for all $X\in\calM$. We must show that $X\otimes L_R S\rightarrow L_R X$ is an equivalence. As $F$ is conservative, it suffices to show that $F(X\otimes L_R S)\rightarrow F(L_R X)$ is an equivalence. As $F$ is symmetric monoidal, and by (2), this is $F(X)\otimes L_{F(R)}F(S)\rightarrow L_{F(R)}F(X)$, which is an equivalence as $F(R)$-localization is smashing.
\end{proof}

Write $\Thick^{\otimes}(R)$ for the thick $\otimes$-ideal of $\calM$ generated by $R$. Following \cite[Section 3]{mathew2015thick}, let $\Tow(\calM)$ denote the category of towers $\{X_n\} = \{\cdots\rightarrow X_1\rightarrow X_0\}$ in $\calM$, let $\Tow^\nil(\calM)$ be the category of towers $\{X_n\}$ for which there exists some $r > 0$ such that all $X_n\rightarrow X_{n-r}$ are null, and let $\Tow^\fast(\calM)$ be the category of towers $\{X_n\}$ for which the associated tower $\Fib\left(\{\lim X_n\}\rightarrow \{X_n\}\right)$ is in $\Tow^\nil(\calM)$, where $\{\lim X_n\}$ is the constant tower on $\lim \{X_n\}$.

\begin{defn}
Say that $R$ is \textit{locally descendable} if $C(R)\in\Tow^\fast(\calM)$.
\tqed
\end{defn}

Given towers $X = \{X_n\}$ and $Y = \{Y_n\}$, write $X\sim Y$ if $\{X_n\}_{n\geq s}\simeq \{Y_n\}_{n\geq s}$ for some $s$. Observe that $X \in \Tow^\fast(\calM)$ if and only if $X\sim C\oplus N$ with $C$ a constant tower and $N\in\Tow^\nil(\calM)$, and in this case $C\simeq \lim X$.

\begin{prop}\label{prop:locdescequiv}
Consider the following conditions:
\begin{enumerate}
\item $R$ is locally descendable, i.e.\ $C(R)\in\Tow^\fast(\calM)$;
\item $L_RA(R) \in \Tow^{\nil}(\calM)$;
\item The map $L_R S \rightarrow C_s(R)$ admits a retraction for some $s$;
\item $R$-localization is smashing and agrees with $R$-nilpotent completion;
\item $L_R S\in \Thick^\otimes(R)$;
\item For all $X\in\calM$, the spectral sequence associated to the tower $\calM(S,X\otimes C(R))$ of spectra collapses at a finite page with a horizontal vanishing line independent of $X$;
\item For all $F\in\calM$ compact, the spectral sequence associated to the tower $\calM(F,C(R))$ of spectra collapses at a finite page with a horizontal vanishing line independent of $F$.
\end{enumerate}
Always (1)$\Leftrightarrow$(2)$\Leftrightarrow$(3)$\Rightarrow$(4,5,6,7). If $R$-localization is smashing, then (5)$\Leftrightarrow$(1,2,3). If $\calM$ is a Brown category \cite[Definition 4.1.4]{hoveypalmieristrickland1997axiomatic}, then (7)$\Leftrightarrow$(1,2,3). If all compact objects in $\calM$ are dualizable, then (6)$\Rightarrow$(7).
\end{prop}
\begin{proof}
Abbreviate $A = A(R)$ and $C = C(R)$ for this proof.

(2)$\Leftrightarrow$(3). As $R$-localization is exact, there is a fiber sequence of towers $L_R A\rightarrow L_R S \rightarrow L_R C$, these localizations taken levelwise. As $C_s\in\Thick^\otimes(R)$ for each $s$, we have $L_R C\simeq C$. Thus there is a fiber sequence $L_R A\rightarrow L_R S \rightarrow C$. Now, if $L_R A\in\Tow^\nil(\calM)$, then $L_R\ol{R}{}^{\otimes s}\rightarrow L_R S$ is null for some $s$, and thus $L_R S \rightarrow C_s$ admits a retraction. Conversely, if $L_R S\rightarrow C_s$ admits a retraction, then $L_R\ol{R}{}^{\otimes s}\rightarrow L_R S$ is null. Any $s$-fold composite $L_R \ol{R}{}^{\otimes n+s}\rightarrow L_R \ol{R}{}^{\otimes n}$ in $L_R A$ is obtained by applying $L_R$ to $\ol{R}{}^{\otimes n}\otimes L_R \ol{R}{}^{\otimes s}\rightarrow \ol{R}{}^{\otimes n}\otimes L_R S$, and must therefore be null, proving $L_R A \in \Tow^\nil(\calM)$.

(2,3)$\Rightarrow$(1). If $L_R S\rightarrow C_s$ admits a retraction, then $C\sim L_R S \oplus L_RA$. As $L_R A\in\Tow^\nil(\calM)$, it follows that $C \in \Tow^\fast(\calM)$.

(1)$\Rightarrow$(4). Suppose that $C\in\Tow^{\fast}(\calM)$. Then $C\sim S_R^\wedge\oplus K$ with $K\in\Tow^\nil(\calM)$. It follows that if $X\in\calM$, then $X_R^\wedge\simeq \lim(X\otimes C)\simeq X\otimes S_R^\wedge \oplus \lim(X\otimes K)\simeq X\otimes S_R^\wedge$, the last equivalence being as $X\otimes K \in \Tow^\nil(\calM)$. Applied to $X = R$, as $R_R^\wedge\simeq R$, we find that $S\rightarrow S_R^\wedge$ is an $R$-equivalence, so that $L_R S \simeq S_R^\wedge$. Combining these gives $X_R^\wedge\simeq X\otimes L_R S$. In particular, $X\otimes L_R S$ is $R$-local, and thus $X\otimes L_R S \simeq L_R X$. Altogether, this proves (4).

(1,4)$\Rightarrow$(3). Suppose $C\in\Tow^{\fast}(\calM)$. Then $C\sim S_R^\wedge\oplus F$ with $F\in\Tow^\nil(\calM)$. By (4), we know $S_R^\wedge \simeq L_R S$, and this implies (3).

(3)$\Rightarrow$(5). This holds as $C_s\in\Thick^\otimes(R)$ for each $s$.

(5)$\Rightarrow$(3) assuming $R$-localization is smashing. One may prove by filtering $\Thick^\otimes(R)$ (cf.\ \cite[Lemma 3.8]{bousfield1979localization} or \cite[Construction 2.5]{mathew2018examples}) that if $X$ is $R$-nilpotent then $X\otimes A \in \Tow^\nil(R)$. In particular, if $L_R S \in \Thick^\otimes(R)$ and $R$-localization is smashing, then $L_R A\simeq L_RS\otimes A\in\Tow^\nil(R)$.

(1)$\Rightarrow$(6,7). These follow from the construction of the spectral sequence of a tower, cf.\ \cite[Proposition 3.12]{mathew2015thick}.

(6)$\Rightarrow$(7) assuming that all compact objects in $\calM$ are dualizable. This holds as $\calM(F,C(R))\simeq\calM(S,DF\otimes C(R))$ for $F$ dualizable, where $DF$ is the dual of $F$.

(7)$\Rightarrow$(1) assuming $\calM$ is a Brown category. Let $K = \Fib(S_R^\wedge\rightarrow C)$, so that $R$ is locally descendable if and only if $K\in\Tow^\nil(\calM)$. By \cite[Proposition 3.12]{mathew2015thick}, condition (7) ensures that there exists some $r > 0$ such that for all $F\in\calM$ compact, all $r$-fold composites in the tower $[F,K]$ of abelian groups vanish. In other words, there exists some $r > 0$ such that all $r$-fold composites in $K$ are phantom maps. \cite[Theorem 4.2.5]{hoveypalmieristrickland1997axiomatic} proves that all composites of phantom maps are nullhomotopic. Thus all $2r$-fold composites in $K$ are nullhomotopic, proving that $K\in\Tow^\nil(\calM)$ and so $C\in\Tow^\fast(\calM)$.
\end{proof}

\begin{cor}\label{cor:locdtrans}
Let $T\in\calM$ be another ring. Suppose that $R\in\Thick^\otimes(T)$ and $\langle R \rangle \subset \langle T \rangle$. If $R$ is locally descendable then $T$ is locally descendable.
\end{cor}
\begin{proof}
As $R \in \Thick^\otimes(T)$, we have $\langle T \rangle\subset\langle R \rangle$. Thus $R$ and $T$ have the same Bousfield class. As $R$ is locally descendable, $R$-localization is smashing. As $R$ and $T$ have the same Bousfield class, it follows that $T$-localization is smashing. As $T$-localization is smashing and $L_T S = L_R S \in \Thick^\otimes(R) \subset\Thick^\otimes(T)$, it follows that $T$ is locally descendable.
\end{proof}

\begin{prop}\label{prop:descpreserve}
Let $\calN$ be another presentably symmetric monoidal stable $\infty$-category, and let $F\colon\calM\rightarrow\calN$ be an exact and symmetric monoidal functor. If $R$ is locally descendable, then $F(R)$ is locally descendable, and $F(L_RX)\simeq L_{F(R)}F(X)$ for any $X\in\calM$.
\end{prop}
\begin{proof}
Suppose that $R$ is locally descendable, and write $C(R)\sim L_R S\oplus K$ with $K\in\Tow^\nil(\calM)$. As $F$ is exact and symmetric monoidal, we have $C(F(R))\simeq F(C(R))\sim F(L_RS)\oplus F(K)$. As $K\in\Tow^\nil(\calM)$ and $F$ is exact, we have $F(K)\in\Tow^\nil(\calN)$. Thus $C(F(R))\in\Tow^\fast(\calN)$, implying that $F(R)$ is locally descendable. Moreover, $L_{F(R)}S\simeq\lim C(F(R))\simeq \lim\left(F(L_RS)\oplus F(K)\right)\simeq  F(L_R S)$. As both $R$-localization and $F(R)$-localization are smashing and $F$ is symmetric monoidal, it follows that $F(L_RX)\simeq F(X\otimes L_RS)\simeq F(X)\otimes L_{F(R)}S\simeq L_{F(R)}F(X)$ for any $X\in\calM$.
\end{proof}

\subsection{Isotropy separation}\label{ssec:isotropy}

Fix a finite group $G$. This section records some techniques that allow one to relate a $G$-spectrum $R$ to its geometric fixed points $\Phi^K R$. We expect that this material is well known to the experts; the reader may observe that the basic approach appears in the proof of the tom Dieck splitting \cite{tomdieck1975orbittypen}, and similar statements appear in \cite[Chapter 2]{lewismaysteinberger1986equivariant} and \cite[Part IV]{greenleesmay1995generalized}. Recently, more sophisticated theorems have appeared which give complete reconstructions of $G$-spectra from their geometric fixed points and appropriate gluing data \cite{glasman2017stratified, ayalamazelgeerozenblyum2021stratified}, though our purposes turn out to be better served by a more elementary approach.

We begin by fixing some notation. We continue to write $\Sp^G$ for the homotopy theory of genuine $G$-spectra. Given a category $\calC$, write $\Fun(BG,\calC)$ for the category of objects in $\calC$ with $G$-action.
Given a subgroup $K\subset G$ and $G$-spectrum $X$, write $\res_K^G X$ for the underlying $K$-spectrum of $X$, and $X^K$ and $\Phi^K X$ for the genuine and geometric $K$-fixed points of $X$. Both $X^K$ and $\Phi^K X$ carry residual actions by the Weyl group $W_G K = N_G(K)/K$, the former via the formula $X^K = \Sp^G(G/K_+,X)$ and the latter as $\Phi^K X$ is a localization of $X^K$. In particular, we may regard $\Phi^K$ as a functor
\[
\Phi^K\colon\Sp^G\rightarrow\Fun(BW_GK,\Sp).
\]

Recall that a \textit{family of subgroups} of $G$ is a collection $\calF$ of subgroups of $G$ closed under subconjugacy. Given such a family, write $\calO_\calF(G)$ for the associated full subcategory of the orbit category of $G$, consisting of those $G$-sets $G/H$ with $H\in\calF$. Associated to any family $\calF$ are two $G$-spaces $E\calF$ and $\widetilde{E}\calF$, which fit into a cofiber sequence
\[
E\calF_+\rightarrow S^0\rightarrow \widetilde{E}\calF,
\]
and are characterized by the fixed points
\[
E\calF_+^K = \begin{cases}
\ast&K\notin\calF,\\
S^0&K\in\calF;
\end{cases}
\qquad
\widetilde{E}\calF^K = \begin{cases}
S^0 & K\notin\calF,\\
\ast&K\in\calF.
\end{cases}
\]
The suspension spectra of these spaces play a central role in equivariant stable homotopy theory; see especially \cite{mathewnaumannnoel2017nilpotence,mathewnaumannnoel2019derived} for a modern account, and \cite[Chapter 7]{dieck1979transformation} for a classical account. We will make use of the following formula for $E\calF$, see \cite[Appendix A]{mathewnaumannnoel2019derived}.

\begin{lemma}
There is an equivalence $
E\calF\simeq\colim_{G/H\in \calO_\calF(G)}G/H.
$
\qed
\end{lemma}

A $G$-spectrum $X$ is said to be \textit{$\calF$-nilpotent} if the map $E\calF_+ \otimes X\rightarrow X$ is an equivalence, and \textit{$\calF^{-1}$-local} if the map $X\rightarrow \widetilde{E}\calF\otimes X$ is an equivalence. An important special case of $\calF^{-1}$-localization is the following, see for instance \cite[Section 6.2]{mathewnaumannnoel2017nilpotence}.

\begin{lemma}\label{lem:gfp}
Let $\calP$ be the family of proper subgroups of $G$. Then
\[
(\widetilde{E}\calP\otimes X)^G\simeq \Phi^G X,
\]
and $\Phi^G$ gives an equivalence from the category of $\calP^{-1}$-local $G$-spectra to the category of ordinary spectra.
\qed
\end{lemma}

An inclusion of families $\calF_1\subset\calF_2$ induces a map $E\calF_{1+}\rightarrow E\calF_{2+}$, and so any $G$-spectrum $X$ may be filtered by the $G$-spectra $E\calF_+\otimes X$. Our main observations in this subsection concern the layers of this filtration. Given families $\calF_1\subset\calF_2$, define
\[
E[\calF_1,\calF_2] = \Cof\left(E\calF_{1+}\rightarrow E\calF_{2+}\right)\simeq \widetilde{E}\calF_1\otimes E\calF_{2+}.
\]
Note that for any $G$-spectrum $X$, there is a natural square
\begin{equation}\label{eq:pair}
\begin{tikzcd}
X\ar[d]&\ar[l]E\calF_{2+}\otimes X\ar[d]\\
\widetilde{E}\calF_1\otimes X&\ar[l]E[\calF_1,\calF_2] \otimes X.
\end{tikzcd}\end{equation}

\begin{lemma}
The square \cref{eq:pair} consists of equivalences if and only if $\Phi^H X \simeq 0$ for all $H\notin \calF_2\setminus\calF_1$. In particular, the homotopy type of $E[\calF_1,\calF_2]$ depends only on $\calF_2\setminus\calF_1$.
\end{lemma}
\begin{proof}
Note that $\Phi^H X\simeq 0$ for all $H\notin\calF_2\setminus\calF_1$ if and only if $\Phi^H X\simeq 0$ for all $H\notin\calF_2$ and all $H\in\calF_1$. The condition that $\Phi^H X\simeq 0$ for all $H\in\calF_1$ is equivalent to $X\rightarrow \widetilde{E}\calF_1\otimes X$ being an equivalence, and the condition that $\Phi^H X \simeq 0$ for all $H\notin\calF_2$ is equivalent to $E\calF_{2+}\otimes X \rightarrow X$ being an equivalence. This shows that if \cref{eq:pair} consists of equivalences, then $\Phi^H X\simeq 0$ for all $H\notin\calF_2\setminus\calF_1$, as well as half of the converse. The other half follows by applying the same argument to $\widetilde{E}\calF_{1}\otimes X$ and $E\calF_{2+}\otimes X$.
\end{proof}

Given a subgroup $K\subset G$, one says that a pair $\calF_1\subset\calF_2$ is \textit{adjacent at $K$} if $\calF_2\setminus\calF_1 = (K)$. In this case, we write $E[K] = E[\calF_1,\calF_2]$; the previous lemma ensures that the homotopy type of $E[K]$ depends only on the conjugacy class of $K$. Say that a $G$-spectrum $X$ is \textit{concentrated at $K$} if $X\simeq E[K]\otimes X$. 

\begin{prop}\label{prop:concentrated}
If $X$ is concentrated at $K$, then
\[
X^G\simeq (\Phi^K X)_{\h W_GK}.
\]
Moreover, $\Phi^K$ defines an equivalence from the full subcategory of $G$-spectra concentrated at $K$ to $\Fun(BW_GK,\Sp)$.
\end{prop}
\begin{proof}
First, note that if $X$ is concentrated at $K$, then $\res^G_K X$ is $\calP_K^{-1}$-local, where $\calP_K$ is the family with respect to $K$ of proper subgroups of $K$. In particular, \cref{lem:gfp} implies that $X^K\simeq\Phi^K X$.

Now, let $\calF_{\leq K}$ be the family of subgroups of $G$ subconjugate to $K$. As $X$ is concentrated at $K$, it is $\calF_{\leq K}$-nilpotent, and thus
\[
X^G\simeq (E\calF_{\leq K} \otimes X)^G \simeq\colim_{G/H\in \calO_{\calF_{\leq K}}(G)} (G/H \otimes X)^G\simeq \colim_{G/H\in\calF_{\leq K}} X^H.
\]
If $H\in\calF_{\leq K}$ is not conjugate to $K$, then the condition that $X$ is concentrated at $K$ implies that $X^H\simeq 0$. This ensures that, though the inclusion $BW_GK\simeq B\Aut(G/K)\subset \calO_{\calF_{\leq K}}(G)$ need not be cofinal as $K\subset G$ need not be normal, this inclusion still induces an equivalence
\[
\colim_{G/H\in\calF_{\leq K}(G)}X^H\simeq\colim_{BW_GK}X^K\simeq (X^K)_{\h W_GK} \simeq (\Phi^K X)_{\h W_GK}.
\]

It remains to verify that $\Phi^K\colon\Sp^G\rightarrow\Fun(BW_GK,\Sp)$ is an equivalence when restricted to the full subcategory of $G$-spectra concentrated at $K$. First we claim that it is fully faithful. Indeed, let $X$ and $Y$ be $G$-spectra concentrated at $K$. Then the same argument as above shows
\[
\Sp^G(X,Y)\simeq\lim_{G/H\in\calO_{\calF_{\leq K}}(G)}\Sp^H(\res^G_HX,\res^G_HY)\simeq\Sp^K(\res^G_KX,\res^G_KY)^{\h W_GK}.
\]
\cref{lem:gfp} implies that $\Sp^K(\res^G_KX,\res^G_KY)\simeq\Sp(\Phi^KX,\Phi^KY)$, and so we have
\[
\Sp^G(X,Y)\simeq\Sp(\Phi^KX,\Phi^KY)^{\h W_GK}.
\]
This is the mapping spectrum in $\Fun(BW_GK,\Sp)$, so that $\Phi^K$ is fully faithful on $G$-spectra concentrated at $K$ as claimed.

Next we claim that it is essentially surjective. As $\Phi^K$ preserves colimits, it suffices to show that if $T$ is a $W_GK$-set then $\Sigma^\infty_+ T \in \Fun(BW_GK,\Sp)$ is in its essential image. To that end, it suffices to produce a pointed $G$-space $X$ satisfying $X^K = T_+$ and $X^H = \ast$ for $H$ not conjugate to $K$, for then $\Sigma^\infty X \in \Sp^G$ is concentrated at $K$ and satisfies $\Phi^K \Sigma^\infty X = \Sigma^\infty_+ T$. Indeed, one easily constructs a presheaf on the orbit category $\calO(G)$ of $G$ satisfying
\[
G/H \mapsto \begin{cases}T_+&H\text{ conjugate to }K,\\\ast&\text{ otherwise},\end{cases}
\]
and with $\Aut_{\calO(G)}(G/K) \cong W_GK$ acting on $T_+$ in the prescribed manner. This then gives rise to the necessary $G$-space by Elmendorf's theorem.
\end{proof}

\begin{lemma}\label{lem:filt}
Any $G$-spectrum $X$ admits a natural finite filtration with
\[
\gr X \simeq \bigoplus_{(K)}E[K]\otimes X,
\]
this sum being over the conjugacy classes of subgroups of $G$.
\end{lemma}
\begin{proof}
Any maximal chain $\calF_0\subset\calF_1\subset\cdots\subset\calF_n$ of families of subgroups of $G$ has the property that each $\calF_i\subset\calF_{i+1}$ is adjacent at some subgroup, and that every conjugacy class appears as $\calF_{i+1}\setminus\calF_i$ for exactly one $i$, so the associated filtration $E\calF_{0+}\otimes X \rightarrow E\calF_{1+}\otimes X\rightarrow\cdots\rightarrow E\calF_{n+}\otimes X$ has the desired properties.
\end{proof}

\begin{prop}\label{cor:filtfix}
Any $G$-spectrum $X$ admits a natural finite filtration with
\[
\gr X^G\simeq\bigoplus_{(K)}(\Phi^K X)_{\h W_GK},
\]
this sum being over the conjugacy classes of subgroups of $G$.
\end{prop}
\begin{proof}
Combine \cref{lem:filt} and \cref{prop:concentrated}.
\end{proof}

\begin{cor}\label{cor:limitdetect}
Let $F\colon \calJ\rightarrow\Sp^G$ be a diagram of $G$-spectra, and $f\colon X\rightarrow \lim_{j\in\calJ}F(j)$ be a map of $G$-spectra. For $f$ to be an equivalence, it suffices that $f$ induces an equivalence
\[
(\Phi^K X)_{\h W_HK}\simeq \lim_{j\in\calJ} \left((\Phi^K F(j))_{\h W_H K}\right)
\]
of ordinary spectra for all subgroups $K\subset H \subset G$.
\end{cor}
\begin{proof}
The map $f$ is an equivalence if and only if it induces an equivalence $f^H\colon X^H\rightarrow \lim_{j\in\calJ} F(j)^H$ for all subgroups $H\subset G$. By \cref{cor:filtfix}, both source and target admit a natural finite filtration, with
\[
\gr f^H \colon \bigoplus_{(K)} (\Phi^K X)_{\h W_H K}\rightarrow \bigoplus_{(K)} \lim_{j\in\calJ} \left((\Phi^K F(j))_{\h W_H K}\right),
\]
these sums being over the conjugacy classes of subgroups $K\subset H$. The corollary follows as $f^H$ is an equivalence provided $\gr f^H$ is an equivalence.
\end{proof}

\subsection{Equivariant Bousfield localizations}

We are now in a good position to discuss equivariant Bousfield localization. Fix a ring $G$-spectrum $R$. Our main observation is the following.

\begin{theorem}\label{thm:desc}
$R$ is locally descendable if and only if each $\Phi^K R$ is locally descendable as an object of $\Fun(BW_GK,\Sp)$. If $W_GK$ acts trivially on $\Phi^K R$, then this holds if and only if $\Phi^K R$ is locally descendable as an ordinary spectrum.
\end{theorem}
\begin{proof}
\cref{prop:descpreserve} implies that if the $G$-spectrum $R$ is locally descendable, then each $\Phi^K R\in\Fun(BW_GK,\Sp)$ is locally descendable; and that if $W_GK$ acts trivially on $\Phi^K R$, then $\Phi^K R$ is locally descendable in $\Fun(BW_GK,\Sp)$ if and only if it is locally descendable in $\Sp$.

Now suppose that each $\Phi^K R\in\Fun(BW_GK,\Sp)$ is locally descendable. By \cref{lem:filt}, $C(R)$ admits a finite filtration with filtration quotients of the form $E[K]\otimes C(R)$. As $\Tow^\fast(\Sp^G)\subset\Tow(\Sp^G)$ is a thick subcategory, it suffices to show that $E[K]\otimes C(R) \in \Tow^\fast(\Sp^G)$ for all $K\subset G$. Under the embedding of \cref{prop:concentrated}, $E[K]\otimes C(R)$ corresponds to the tower $C(\Phi^K R)\in \Tow(\Fun(BW_GK,\Sp))$, so this follows from the assumption that $\Phi^K R$ is locally descendable in $\Fun(BW_GK,\Sp)$.
\end{proof}

We can extend this to the global equivariant context. First, some notation. For our purposes, a \textit{global family} shall be a collection $\calF$ of finite groups closed under products, subgroups, and quotients. Schwede \cite{schwede2018global} has shown that to each global family $\calF$, there is a good symmetric monoidal and stable category $\Glob_\calF$ of global spectra with respect to $\calF$.

Associated to any $G\in\calF$ is a symmetric monoidal functor
\[
U_G\colon\Glob_\calF\rightarrow\Sp^G,
\]
which preserves limits and colimits \cite[Theorem 4.5.25]{schwede2018global}. Moreover, these functors are jointly conservative as $G$ is taken to range through $\calF$, and are compatible with each other in the sense that $\res^G_K U_G = U_K$ for $K\subset G$.

Associated to any $X\in\Glob_\calF$ and $G\in\calF$ are the genuine and geometric fixed points $X^G$ and $\Phi^G X$. The genuine fixed points $X^G$ are represented by the global suspension spectrum of the global classifying space $B_{\gl}G$, in the sense that $X^G\simeq\Glob_\calF(B_{\gl}G_+,X)$ \cite[Theorem 4.4.3]{schwede2018global}. In particular, $X^G$ carries a natural action by the space $\Aut(B_{\gl}G)$ of automorphisms of the global classifying space $B_{\gl}G$. This in turn is equivalent to the space of automorphisms of the ordinary classifying space $BG$, as can be easily seen from the orbispace model for global spaces \cite{korschgen2018comparison}.

Genuine and geometric fixed points are compatible with the functors $U_G$, in the sense that $(U_G)^K \simeq X^K$ and $\Phi^K U_G X \simeq \Phi^K X$. Following the discussion after \cite[Theorem 4.5.25]{schwede2018global}, if we write $L$ for the left adjoint to $U_G$, then the natural equivalences $\Glob_\calF(B_{\gl}K,X)\simeq X^K\simeq\Sp^G(G/K_+,U_GX)\simeq\Glob_\calF(L(G/K)_+,X)$ show that $L(G/K_+) \simeq B_{\gl}K$ for $K\subset G$. It follows that $W_GK$ acts on $X^K$ through its action on $BK \simeq EG\times_G (G/K)$.

\begin{theorem}\label{thm:globdesc}
Let $\calF$ be a global family, and suppose that for all $G\in\calF$ and $K\subset G$, the spectrum $\Phi^K R$ is locally descendable as an object of $\Fun(BW_GK,\Sp)$. Then $R$-localization is smashing and agrees with $R$-nilpotent completion, $U_G R\in\Sp^G$ is locally descendable for all $G\in\calF$, and $U_G L_R X\simeq L_{U_G R}U_G X$ for all $X\in\Glob_\calF$.
\end{theorem}
\begin{proof}
The hypotheses ensure that we may apply \cref{thm:desc} to deduce that $U_G R \in \Sp^G$ is locally descendable for all $G\in\calF$. The remaining assertions follow by applying \cref{lem:nilreflect} to $(U_G)_{G\in\calF}\colon \Glob_\calF\rightarrow\prod_{G\in\calF} \Sp^G$.
\end{proof}

In general, it seems difficult to determine when a ring $R\in\Fun(BG,\Sp)$ is locally descendable when $G$ acts nontrivially on $R$. We will make use of the following simple case.

\begin{lemma}\label{lem:borelaway}
Let $G$ be a finite group and $R\in\Fun(BG,\Sp)$ be a ring. If $|G|$ acts invertibly on $R$ and the ordinary spectrum $R^{\h G}$ is locally descendable, then $R$ is locally descendable.
\end{lemma}
\begin{proof}
As there is a $G$-equivariant map $i\colon R^{\h G}\rightarrow R$ of rings, we have $\langle R^{\h G}\rangle \subset \langle R \rangle$. As $G$ acts invertibly on $R$, the composite $R^{\h G}\rightarrow R\rightarrow R_{\h G}\rightarrow R^{\h G}$, with last map the transfer, is an equivalence. Thus $R^{\h G}\in\Thick^{\otimes}(R)$, and the lemma then follows from \cref{cor:locdtrans}.
\end{proof}

We also need the following.

\begin{prop}\label{prop:locpresg}
Suppose that $R$ is a $G\hyp\bbE_\infty$ ring. Then $R$-nilpotent completion preserves $G\hyp\bbE_\infty$ rings. In particular, if all $G$-spectra are $R$-convergent, such as if $R$ is locally descendable, then $R$-localization preserves $G\hyp\bbE_\infty$ rings. The same statements hold with $G\hyp\bbE_\infty$ ring spectra replaced by global ultracommutative ring spectra.
\end{prop}
\begin{proof}
If $R$ is an $\bbA_\infty$ ring, then $C(R)$ may be identified as the tower of partial totalizations of the cosimplicial object $[n]\mapsto R^{\otimes n+1}$ \cite[Proposition 2.14]{mathewnaumannnoel2017nilpotence}, and thus $R$-nilpotent completion is given by $X_R^\wedge = \lim_{n\in\Delta} (X\otimes R^{n+1})$. When moreover $R$ and $X$ are $G\hyp\bbE_\infty$ rings, this is the totalization of a cosimplicial diagram of $G\hyp\bbE_\infty$ rings, and is therefore itself a $G\hyp\bbE_\infty$ ring. The same proof applies in the global ultracommutative case.
\end{proof}

So far we have focused on localizations with particularly good finiteness properties. We also note an orthogonal case. First, a bit more notation. The forgetful functors $U\colon \Sp^G\rightarrow\Fun(BG,\Sp)$ and $U\colon \Glob_\calF\rightarrow\Sp$ admit right adjoints, which we shall denote $b_G$ and $b_\calF$ respectively. In particular, $b_GU(X)\simeq F(EG_+,X)$, where $EG_+$ is the classifying space for the family $\{e\}$.

\begin{prop}\label{prop:borelloc}
Let $T$ be an ordinary ring spectrum.
\begin{enumerate}
\item If $T^{tG} = 0$, then $L_{b_G(T)}X\simeq b_G(L_TUX)$ for all $X\in\Sp^G$.
\item If $T^{tG}=0$ for all $G\in\calF$, then $L_{b_\calF(T)}X\simeq b_\calF(L_TUX)$ for all $X\in\Glob_\calF$.
\end{enumerate}
\end{prop}
\begin{proof}
The proof is essentially the same in both cases, so we shall just prove the first. The assumption  that $T$ is a ring and $T^{tG} = 0$ implies that $\Phi^K b_G(T) = 0$ for all nontrivial subgroups $K\subset G$ \cite[Proposition 2.13]{mathewnaumannnoel2019derived}. At this point, we could deduce (1) by observing that $b_G(T)$ is Bousfield equivalent to $G_+\otimes T$ and applying \cite[Proposition 3.21]{carrick2020smashing}; however, we shall give the direct proof that also applies in case (2).

First we show that $b_G(L_T UX)$ is $b_G(T)$-local. Fix $C\in\Sp^G$ which is $b_G(T)$-acyclic. As $U(b_G(T)\otimes C)\simeq T\otimes UC$, it follows that $UC$ is $T$-acylic. Thus
\[
\Sp_G(C,b_G(L_T UX))\simeq \Sp(UC,L_T UX)^{\h G}\simeq 0,
\]
and this implies that $b_G(L_T UX)$ is $b_G(T)$-local as claimed.

Next we show that $X\rightarrow b_G(L_TUX)$ is a $b_G(T)$-equivalence. To that end, we must show that the map 
\begin{equation}\label{eq:compr}
b_G(T)\otimes X\rightarrow b_G(T)\otimes b_G(L_T UX)
\end{equation}
is an equivalence. It suffices to verify this after applying $\Phi^K$ for all $K\subset G$. If $K = e$, then $\Phi^e = U$ and \cref{eq:compr} is the equivalence $T\otimes UX\rightarrow T\otimes L_T UX$. If $K\neq e$, then both sides of \cref{eq:compr} vanish as $\Phi^K$ is symmetric monoidal and $\Phi^K b_G(T) = 0$ for $K\neq e$.

Together these prove that $X\rightarrow b_G(L_T UX)$ realizes $b_G(L_T UX)$ as the $b_G(T)$-localization of $X$.
\end{proof}

\subsection{Examples}\label{ssec:locexamples}

We now give examples, beginning with the proof of \cref{prop:kul}. Recall that $\textbf{KU}$ denotes the global spectrum of equivariant $K$-theory \cite[Section 6.4]{schwede2018global}, satisfying $U_G \textbf{KU} \simeq KU_G$ for all $G$. We need the following.

\begin{lemma}[{\cite[Section 7.7]{dieck1979transformation}}]\label{lem:geofixed}
For any group $G$, we have
\[
\Phi^G \textbf{KU}\simeq \begin{cases}KU[\tfrac{1}{n}](\zeta_n)&G\cong C_n;\\0&\text{otherwise}.
\end{cases}
\]
\qed
\end{lemma}

Given a subgroup $K\subset G$, define $V_GK = \im(N_GK\rightarrow\Aut(K))$; we comment that $|V_GK| = [N_GK:C_GK]$ where $C_GK$ is the centralizer of $K$ in $G$. Say that $G$ is \textit{$KU$-allowable} if for all cyclic subgroups $C\subset G$, the order of $V_GC$ is invertible in $\bbZ[\tfrac{1}{|C|}]$.

\begin{theorem}\label{prop:kuall}
$KU_G\in\Sp^G$ is locally descendable if and only if $G$ is $KU$-allowable.
\end{theorem}
\begin{proof}
Suppose that $G$ is $KU$-allowable. By \cref{thm:desc}, we must show that $\Phi^K KU_G \in \Fun(BW_GK,\Sp)$ is locally descendable for all subgroups $K\subset G$. By \cref{lem:geofixed}, we need only consider the case where $K = C$ is a cyclic subgroup of order $n$. Here $\Phi^C KU_G = KU[\tfrac{1}{n}](\zeta_n)$ is an $\Aut(C)$-Galois extension of $KU[\tfrac{1}{n}]$, and the Weyl group $W_GC$ acts on $KU[\tfrac{1}{n}](\zeta_n)$ through a natural $\Aut(BC)$-action extending its $\Aut(C)$-action. As $\Aut(BC)$ is $1$-truncated and the order of $\pi_1 \Aut(BC) = C$ is invertible in $KU[\tfrac{1}{n}](\zeta_n)$, the $\Aut(BC)$-action on $KU[\tfrac{1}{n}](\zeta_n)$ factors through the truncation $\Aut(BC)\rightarrow\pi_0 \Aut(BC)\cong\Aut(C)$. Thus $W_GC$ acts on $KU[\tfrac{1}{n}](\zeta_n)$ through the natural map $W_GC\rightarrow V_GC$, and it suffices to show that $KU[\tfrac{1}{n}](\zeta_n)\in\Fun(BV_GC,\Sp)$ is locally descendable.

By assumption, the order of $V_GC$ is invertible in $KU[\tfrac{1}{n}](\zeta_n)$, so by \cref{lem:borelaway} it suffices to show that the ordinary spectrum $KU[\tfrac{1}{n}](\zeta_n)^{\h V_GC}$ is locally descendable. This assumption moreover implies that $\pi_\ast (KU[\tfrac{1}{n}](\zeta_n)^{\h V_GC})\cong (\pi_\ast KU[\tfrac{1}{n}](\zeta_n))^{V_GC}$; this is in particular a free $\pi_\ast KU[\tfrac{1}{n}]$-module, and thus $KU[\tfrac{1}{n}](\zeta_n)^{\h V_GC}$ is a free $KU[\tfrac{1}{n}]$-module. Hence by \cref{cor:locdtrans} it suffices to verify that $KU[\tfrac{1}{n}]$ is locally descendable. This is the classical example of a locally descendable spectrum: \cite[Corollary 4.7]{bousfield1979localization} shows that $KU$-localization is smashing and $L_{KU}S\in \Thick^\otimes(KU)$, so the same is true for $KU[\tfrac{1}{n}]$, and local descendability then follows from \cref{prop:locdescequiv}.

Now suppose that $G$ is not $KU$-allowable. We may thus find a cyclic subgroup $C\subset G$ of order $n$, prime $p$ not dividing $n$, and cyclic $p$-subgroup $D\subset N_GC$ for which the composite $D\rightarrow N_GC\rightarrow\Aut(C)$ is nonzero. Write $\Phi^C KU_G = KU[\tfrac{1}{n}](\zeta_n)$. By \cref{prop:descpreserve}, to show that $KU_G\in\Sp^G$ is not locally descendable it suffices to show that $KU[\tfrac{1}{n}](\zeta_n)\in\Fun(BD,\Sp)$ is not locally descendable. In the following, abbreviate $L = L_{KU[\frac{1}{n}](\zeta_n)}$.

For a spectrum $X$ write $i(X)\in\Fun(BD,\Sp)$ for the corresponding object with trivial action. Then $i(S)$ is the unit of $\Fun(BD,\Sp)$, so by \cref{prop:locdescequiv} it suffices to show that $Li(S)\notin\Thick^\otimes(KU[\tfrac{1}{n}](\zeta_n))$. Observe that we may additively identify $KU[\tfrac{1}{n}](\zeta_n)\simeq \Aut(C)_+\otimes KU[\tfrac{1}{n}]$. As the image of $D$ in $\Aut(C)$ is nontrivial, it follows that if $X\in\Thick^\otimes(KU[\tfrac{1}{n}](\zeta_n))$ then $\Phi^D b_D(X) = 0$, so it suffices to verify that $\Phi^D b_D(Li(S)) \neq 0$.

Observe that $KU[\tfrac{1}{n}](\zeta_n)\in\Fun(BD,\Sp)$ has the same Bousfield class as $i(KU[\tfrac{1}{n}])$. By  \cref{prop:descpreserve}, as $KU[\tfrac{1}{n}]$ is locally descendable, we find $Li(S)\simeq L_{i(KU[\frac{1}{n}])}i(S)\simeq i(L_{KU[\frac{1}{n}]}S)$. As $D$ is a cyclic $p$-group and $p\nmid n$, it is easily verified that $\Phi^D b_D(i(L_{KU[\frac{1}{n}]}S))\neq 0$, see for instance \cite[Proposition 5.36]{mathewnaumannnoel2019derived}, and this finishes the proof.
\end{proof}

A good supply of $KU$-allowable groups is given by the following.

\begin{lemma}\label{lem:nilallow}
Suppose that $G$ is nilpotent. Then $G$ is $KU$-allowable.
\end{lemma}
\begin{proof}
As $G$ is a finite nilpotent group, we may write $G = \prod_p G_{(p)}$ with $G_{(p)}\subset G$ the Sylow $p$-subgroup. It follows that if $C\subset G$ is any subgroup, then $C = \prod_{p}C_{(p)}$ with $C_{(p)} = C \cap G_{(p)}$, and that $V_GC = \prod_p V_{G_{(p)}}C_{(p)}$. Thus if a prime $p$ divides the order of $V_GC$, then $V_{G_{(p)}}C_{(p)}\neq e$, implying that $C_{(p)} \neq e$ and thus that $p$ divides the order of $C$. As every prime dividing the order of $V_GC$ divides the order of $C$, we find that the order of $V_GC$ is invertible in $\bbZ[\tfrac{1}{|C|}]$, and so $G$ is $KU$-allowable as claimed.
\end{proof}

The following now suffices to prove \cref{prop:kul}.

\begin{prop}\label{prop:ku1proof}
Let $\calF$ be a family of groups, all of which are $KU$-allowable.
\begin{enumerate}
\item Bousfield localization in $\Glob_\calF$ with respect to $\textbf{KU}$ is smashing, agrees with nilpotent completion, and preserves ultracommutative ring spectra;
\item If $G$ is $KU$-allowable, then $KU_G\in\Sp^G$ is locally descendable and $KU_G$-localization preserves $G\hyp\bbE_\infty$ ring spectra;
\item $U_G L_{\textbf{KU}}X\simeq L_{KU_G}U_G X$ for all $G\in\calF$ and $X\in\Glob_\calF$.
\end{enumerate}
\end{prop}
\begin{proof}
Given \cref{prop:kuall}, these follow from \cref{thm:desc}, \cref{thm:globdesc}, and \cref{prop:locpresg}.
\end{proof}

At this point, we have provided everything needed in the body of the paper. The remainder of the appendix is dedicated to giving some additional examples of the theory developed above. We start by noting that the techniques of \cref{ssec:isotropy} may be used to give more quantitative information about equivariant $K$-theory localizations.

\begin{prop}\label{prop:kulocid}
Let $G$ be a $KU$-allowable group and let $X$ be a $G$-spectrum. Then for $K\subset G$, we may identify
\[
\Phi^K L_{KU_G}X\simeq\begin{cases} L_{KU[\frac{1}{n}]}\left(\Phi^K X\right)&K\cong C_n,\\0&\text{otherwise},
\end{cases}
\]
and $L_{KU_G}X$ admits a finite filtration with
\[
\gr (L_{KU_G}X)^G \simeq \bigoplus_{(C) \text{ cyclic}} L_{KU[\tfrac{1}{|C|}]}(\Phi^C X)_{\h W_GC},
\]
this sum being over conjugacy classes of cyclic subgroups of $G$.
\end{prop}
\begin{proof}
The identification of $\Phi^K L_{KU_G}X$ follows from \cref{prop:descpreserve} and \cref{lem:geofixed}, and the filtration from \cref{cor:filtfix}.
\end{proof}

If $G$ is a $p$-group, then $(KU_G)_{(p)}$ is Bousfield equivalent to $b_G(KU_p)$. This suggests looking at $b_G(E_n)$-localization, where $E_n$ is a height $n$ Morava $E$-theory, as a higher chromatic analogue of $(KU_G)_{(p)}$-localization. We start by recalling the smash product theorem \cite[Chapter 7]{ravenel1992nilpotence} in its strong form. Say that a Landweber exact spectrum $R$ is of finite height if there exists some $n \geq 0$ such that $R_\ast/(p,v_1,\ldots,v_n) = 0$ at all primes $p$.

\begin{lemma}\label{lem:lexdesc}
Let $R$ be a Landweber exact ring spectrum of finite height. Then $R$ is locally descendable.
\end{lemma}
\begin{proof}
This is clear if $R\simeq 0$, so we may suppose that $R$ is nonzero. By \cref{prop:locdescequiv}, it suffices to show that for all spectra $X$, the $R$-based Adams spectral sequence for $X$ collapses at a finite page with a horizontal vanishing line which is independent of $X$. As $p$-localization is exact, we may identify the $R_{(p)}$-based Adams spectral sequence for $X$ as the $p$-localization of the $R$-based Adams spectral sequence for $X$. It therefore suffices to show that the $R_{(p)}$-based Adams spectral sequence for $X$ collapses at a finite page with a horizontal vanishing line which is independent of $X$, as well as of $p$ for all sufficiently large primes $p$.

Fix a prime $p$, and let $m\leq n$ be maximal for which $R_\ast/(p,v_1,\ldots,v_{m-1})\neq 0$. Applying the theory of \cite{hoveystrickland2003comodules} to the zigzag $R_{(p)}\rightarrow R_{(p)} \otimes E(m)\leftarrow E(m)$, we find that the $R_{(p)}$-based Adams spectral sequence is isomorphic to the $E(m)$-based Adams spectral sequence from the $E_2$ page on.

If $p > m + 1$, then the $E(m)$-based Adams spectral sequence has a horizontal vanishing line on the $E_2$-page of $y$-intercept at most $m^2 + m$ \cite[Theorem 5.1]{hoveysadofsky1999invertible}. This gives a horizontal vanishing line in the $R_{(p)}$-based Adams spectral sequence for $p > n + 1$ which is independent of such $p$.

It now suffices to show that the $R_{(p)}$-based Adams spectral sequence has some horizontal vanishing line for each of the finitely many primes $p \leq n + 1$. As above, we may replace the $R_{(p)}$-based Adams spectral sequence with the $E(m)$-based Adams spectral sequence. The lemma then follows from \cite[Proposition 6.5]{hoveystrickland1999morava}.
\end{proof}

\begin{lemma}\label{lem:philex}
Let $G$ be a finite group. Let $R$ be a $G$-ring spectrum which admits Thom isomorphisms for complex representations, and suppose moreover that $R^G$ is a Landweber exact ring spectrum of finite height. Then the ordinary spectrum $\Phi^G R$ is Landweber exact and locally descendable.
\end{lemma}
\begin{proof}
As $R$ admits Thom isomorphisms for complex representations, we may identify $\Phi^G R \simeq R^G[e^{-1}]$ where $e$ is the oriented Euler class of the reduced complex regular representation of $G$, see \cite[Section 7.4]{dieck1979transformation} or \cite[Section 5]{mathewnaumannnoel2019derived}. As localization is exact, it follows that $\Phi^G R$ is Landweber exact and finite height, so we may conclude by \cref{lem:lexdesc}.
\end{proof}

We now consider the spectra $b_G(E_n)$. We focus on the case where $G$ is an elementary abelian $p$-group, as here all Weyl groups act trivially. We expect that the following observations extend to all abelian $p$-groups, but proving this would require developing additional techniques for determining when an object of $\Fun(BG,\Sp)$ is locally descendable.

\begin{prop}\label{prop:boreldesc}
Let $B$ be an elementary abelian $p$-group. Then $b_B(E_n)$ is locally descendable and $b_B(E_n)$-localization preserves $B\hyp\bbE_\infty$ rings.
\end{prop}
\begin{proof}
By \cref{thm:desc}, to show that $b_B(E_n)$ is locally descendable it suffices to show that the ordinary spectrum $\Phi^A b_B(E_n)$ is locally descendable for all $A\subset B$. Indeed, $b_B(E_n)^A = E_n^{BA_+}$ is a free $E_n$-module, and therefore \cref{lem:philex} applies. That $b_B(E_n)$-localization preserves $B\hyp\bbE_\infty$ rings now follows from \cref{prop:locpresg}.
\end{proof}

As with equivariant $K$-theory, it is possible to be more explicit. In the following, we take the convention that $E_0 = H\bbQ$ and $E_n = 0$ for $n < 0$. Given an elementary abelian $p$-group $A$, write $\rk(A)$ for the rank of $A$, i.e.\ the dimension of $A$ viewed as a vector space over $\bbF_p$.

\begin{lemma}[\cite{torii2002geometric}]\label{lem:heightshift}
Let $A$ be an abelian $p$-group of rank $t$. Then $\langle \Phi^A b_A(E_n)\rangle = \langle E_{n-t}\rangle$.
\end{lemma}
\begin{proof}
As $\Phi^A b_A(E_n)$ is $p$-local and Landweber exact, it is Bousfield equivalent to $E_d$ where $d$ is maximal for which $\Phi^A b_A(E_n)/(v_0,\ldots,v_{d-1})\neq 0$. As $\Phi^A b_A(E_n)/(v_0,\ldots,v_{d-1})\simeq\Phi^A b_A(E_n/(v_0,\ldots,v_{d-1}))$, \cite[Proposition 5.28]{mathewnaumannnoel2019derived} says that this is nonzero if and only if $t \leq n - d$, i.e.\ $d \leq n - t$, proving the lemma.
\end{proof}

Abbreviate $L_n = L_{E_n}$ and $L_n^b = L_{b_B(E_n)}$.

\begin{prop}\label{prop:bloc}
Let $B$ be an elementary abelian $p$-group and let $X$ be a $B$-spectrum. Then
\[
\Phi^A L_n^b X \simeq L_{n-\rk(A)}\Phi^A X,
\]
and $L_n^b X$ admits a finite filtration with
\[
\gr (L_n^b X)^B \simeq \bigoplus_{A\subset B} L_{n-\rk(A)}(\Phi^A X)_{\h B/A},
\]
this sum being over the subgroups of $B$.
\end{prop}
\begin{proof}
The identification of $\Phi^A L_n^bX$ follows from \cref{prop:descpreserve} and \cref{lem:heightshift}, and the filtration from \cref{cor:filtfix}.
\end{proof}

We end our discussion of the localizations $L_n^b$ with the following observation.

\begin{prop}
Let $B$ be an elementary abelian $p$-group. Then
\[
(S_B)_{(p)}\simeq \lim_{n\rightarrow\infty}L_n^b S_B.
\]
\end{prop}
\begin{proof}
By \cref{cor:limitdetect} and \cref{prop:bloc}, the map $(S_B)_{(p)}\rightarrow\lim_{n\rightarrow\infty}L_n^b S_B$ is an equivalence provided the following condition holds. Let $D\subset C \subset B$ be subgroups, write $A = W_CD = C/D$, and suppose that $C$ is of rank $t$. Then the map
\[
(BA_+)_{(p)}\rightarrow\lim_{n\rightarrow\infty} L_{n-t}BA_+
\]
is an equivalence. This in turn holds if and only if $S_{(p)}\simeq\lim_{n\rightarrow\infty} L_n S$ and $\Sigma^\infty BA\simeq \lim_{n\rightarrow\infty}L_n \Sigma^\infty BA$. The first condition is exactly the classical chromatic convergence theorem \cite[Theorem 7.5.7]{ravenel1992nilpotence}. The second condition asks that chromatic convergence holds for $\Sigma^\infty BA$.

Barthel \cite{barthel2016chromatic} has shown that if $X$ is a spectrum, then $\lim_{n\rightarrow\infty}L_nX \simeq X_{(p)}$ provided that $X$ has finite projective $BP$-dimension. Work of Johnson--Wilson \cite{johnsonwilson1985brown} shows that if $A$ is an elementary abelian $p$-group, then $BA$ has projective $BP$-dimension equal to $\rk(A)$. Combining these proves the proposition.
\end{proof}

We end with an orthogonal class of examples. Combining \cref{prop:borelloc} with \cite[Theorem 1.1]{greenleessadofsky1996tate} shows that $L_{b_G(K(n))} X\simeq b_G(L_{K(n)}UX)$ for any finite group $G$ and $G$-spectrum $X$. Let us just focus on the simplest case, which may be interpreted more conceptually.

\begin{prop}\label{prop:pk1loc}
Let $G$ be a $p$-group. Then for any $X\in\Sp^G$, we have
\[
L_{KU_G/(p)}X\simeq b_G(L_{KU/(p)} UX).
\]
\end{prop}
\begin{proof}
By the $p$-adic version of Atiyah's completion theorem \cite[Theorem 7.2]{atiyah1961characters} \cite[III \S 1, Proposition 1.1]{atiyahtall1969group}, there is an equivalence $KU_G/(p)\simeq b_G(KU/(p))$. By \cite[Theorem 13.1]{greenleesmay1995generalized}, $(KU/(p))^{t G} = 0$ for any finite group $G$. The proposition then follows from \cref{prop:borelloc}.
\end{proof}

We deduce the following corollary, which was also independently obtained in \cite[Proposition 6.3]{bonventreguilloustapleton2022kug} for $p$ odd.

\begin{cor}
If $G$ is a $p$-group, $A$ is a finite $G$-spectrum, and $k\in\bbZ_p^\times$ projects to a topological generator of $\bbZ_p^\times/\{\pm 1\}$, then there is a fiber sequence
\begin{equation}\label{eq:kufiber}
\begin{tikzcd}
L_{KU_G/(p)}A\ar[r]&(KO_G\otimes A)_p^\wedge\ar[r,"\psi^k-1"]&(KO_G\otimes A)_p^\wedge
\end{tikzcd}.\end{equation}
If $p$ is odd, then we may replace $KO$ by $KU$ provided $k$ is a topological generator of $\bbZ_p^\times$.
\end{cor}
\begin{proof}
The assumptions that $G$ is a $p$-group and $A$ is finite ensure that \cref{eq:kufiber} is equivalent to
\begin{center}\begin{tikzcd}[column sep=large]
b_G(L_{KU/(p)} UA)\ar[r]& b_G( (KO\otimes A)_p^\wedge)\ar[r,"b_G(\psi^k-1)"]&b_G((KO\otimes A)_p^\wedge)
\end{tikzcd};\end{center}
in other words, that \cref{eq:kufiber} is the image of the standard fiber sequence for $L_{KU/(p)}UA$ under the functor $b_G$. The corollary follows as $b_G$ is exact.
\end{proof}

\begin{rmk}\label{rmk:k1norm}
Let us relate \cref{prop:pk1loc} to the body of the paper. If $R$ is an $\bbE_\infty$ ring and $G$ is a finite group, then $b_G(R)$ is a $G\hyp\bbE_\infty$ ring. If $K\subset G$ is a subgroup of index $m$ and $\alpha\in RO(K)$, then the norm
\[
P_\alpha\colon \pi_\alpha^K b_G(R)\rightarrow \pi_{\Ind_K^G\alpha}^G b_G(R),
\]
may be identified as the composite
\begin{align*}
[\Th(\alpha\downarrow BK),S_{K(1)}]&\rightarrow [\Th(\alpha\downarrow BK)^{\otimes m}_{\h \Sigma_m},S_{K(1)}] \\
&\cong [\Th((\rho_m\otimes\alpha)\downarrow B(\Sigma_m\wr K)),S_{K(1)}]\rightarrow[\Th((\Ind_K^G\alpha)\downarrow BG),S_{K(1)}],
\end{align*}
where the first map is an ordinary power operation and the last map is restriction along a suitable map $BG\rightarrow B(\Sigma_m\wr K)$.

In particular, take $G = C_p$, and suppose that $R$ is $p$-local. As the map
\[
(S^n)^{\otimes p}_{\h C_p}\rightarrow (S^n)^{\otimes p}_{\h \Sigma_p}
\]
is $p$-locally the projection onto a summand, norms for $b_{C_p}(R)$ are determined by the $p$th symmetric powers for $R$ discussed in \cref{ssec:powprelim}. In light of \cref{prop:pk1loc}, we may therefore regard our computation in \cref{sec:k1} as describing norms on $L_{KU_{C_p}/(p)}S_{C_p}$, although to make this completely explicit would require describing the effect of the projection $(S^n)^{\otimes p}_{\h C_p}\rightarrow (S^n)^{\otimes p}_{\h \Sigma_p}$ on $K(1)$-local cohomotopy.
\tqed
\end{rmk}

\end{appendix}

\begingroup
\raggedright

\bibliography{refs}
\bibliographystyle{alpha}

\endgroup

\end{document}